\documentclass[11pt]{amsart}
\usepackage{amsthm,amsmath,amstext,amssymb,amscd,epsfig,euscript, mathrsfs, dsfont,pspicture,multicol,graphpap,graphics,graphicx,times,enumerate,subfig,sidecap,wrapfig}
\usepackage[dvipsnames]{xcolor}
\usepackage{pgf}
\usepackage{stackrel}
\usepackage[colorlinks,citecolor=red]{hyperref}

\hypersetup{colorlinks=true, linkcolor=black}

\usepackage{enumitem}
\makeatletter
\def\namedlabel#1#2{\begingroup
    #2%
    \def\@currentlabel{#2}%
    \phantomsection\label{#1}\endgroup
}
\makeatother

\newcommand{\C}{{\mathbb C}}       % Field of complex numbers
\newcommand{\R}{{\mathbb R}}       % Field of real numbers
\newcommand{\Z}{{\mathbb Z}}       % Ring of integer numbers

\newcommand{\DD}{{\mathcal D}}
\newcommand{\HH}{{\mathcal H}}

\newcommand{\EE}{{\mathcal E}}
\newcommand{\WW}{{\mathcal W}}

\newcommand{\TT}{{\mathsf T}}

\newcommand{\diam}{\operatorname{diam}}
\newcommand{\dist}{{\rm dist}}

\newcommand{\fiproof}{{\hspace*{\fill} $\square$ \vspace{2pt}}}

\newcommand{\rf}[1]{{\eqref{#1}}}
\newcommand{\supp}{\operatorname{supp}}

\newcommand{\vphi}{{\varphi}}
\newcommand{\ve}{{\varepsilon}}
\newcommand{\vv}{{\vspace{2mm}}}
\newcommand{\vvv}{{\vspace{3mm}}}
\newcommand{\wt}[1]{{\widetilde{#1}}}
\newcommand{\wh}[1]{{\widehat{#1}}}

\newcommand{\noi}{\noindent}

\newcommand{\Rd}{\color{red}}

\newcommand{\fG}{{\mathsf G}}
\newcommand{\LD}{{\mathsf{LD}}}
\newcommand{\HD}{{\mathsf{HD}}}

\newcommand{\fS}{{\mathsf{S}}}
\newcommand{\fH}{{\mathsf{H}}}

\newcommand{\sss}{{\mathsf{Stop}}}
\newcommand{\ssss}{{\mathsf{SubStop}}}

\newcommand{\ttt}{{\mathsf{Top}}}

\newcommand{\tree}{{\mathsf{Tree}}}

\newcommand{\stree}{{\mathsf{SubTree}}}

\newcommand{\eend}{{\mathsf{End}}}

\newcommand{\bdy}{{\mathsf{Bdy}}}

\newcommand{\WA}{{{\mathsf{WA}}}}
\newcommand{\WSBC}{{{\mathsf{WSBC}}}}

\newcommand{\Con}{{\mathrm{Con}}}
\newcommand{\NC}{{\mathsf{NC}}}

\newtheorem{theorem}{Theorem}[section]
\newtheorem*{theorem*}{Theorem}
\newtheorem*{lemma*}{Lemma}
\newtheorem*{theorema*}{Theorem A}
\newtheorem*{theoremb*}{Theorem B}
\newtheorem*{theoremc*}{Theorem C}
\newtheorem*{mainlemma*}{Main Lemma}

\newtheorem{lemma}[theorem]{Lemma}
\newtheorem{mlemma}[theorem]{Main Lemma}
\newtheorem{coro}[theorem]{Corollary}

\theoremstyle{definition}

\def\Xint#1{\mathchoice		%average integral
{\XXint\displaystyle\textstyle{#1}}%
{\XXint\textstyle\scriptstyle{#1}}%
{\XXint\scriptstyle\scriptscriptstyle{#1}}%
{\XXint\scriptscriptstyle\scriptscriptstyle{#1}}%
\!\int}
\def\XXint#1#2#3{{\setbox0=\hbox{$#1{#2#3}{\int}$ }
\vcenter{\hbox{$#2#3$ }}\kern-.58\wd0}}

\def\avint{\Xint-}

\theoremstyle{remark}
\newtheorem{remark}[theorem]{\bf Remark}

\numberwithin{equation}{section}

\begin{document}

\title[Harmonic measure and quantitative connectivity. Part II]{Harmonic measure and quantitative connectivity: geometric characterization of the $L^p$ solvability of the Dirichlet problem. Part II}

\author[Azzam]{Jonas Azzam}

\address{Jonas Azzam\\
School of Mathematics \\ University of Edinburgh \\ JCMB, Kings Buildings \\
Mayfield Road, Edinburgh,
EH9 3JZ, Scotland.}
\email{j.azzam "at" ed.ac.uk}

\author[Mourgoglou]{Mihalis Mourgoglou}

\address{Mihalis Mourgoglou\\
Departamento de Matem\'aticas, Universidad del Pa\' is Vasco, Barrio Sarriena s/n 48940 Leioa, Spain and\\
Ikerbasque, Basque Foundation for Science, Bilbao, Spain.
}
\email{michail.mourgoglou@ehu.eus}

\author[Tolsa]{Xavier Tolsa}
\address{Xavier Tolsa
\\
ICREA, Passeig Llu\'{\i}­s Companys 23 08010 Barcelona, Catalonia, and\\
Departament de Matem\`atiques and BGSMath
\\
Universitat Aut\`onoma de Barcelona
\\
Edifici C Facultat de Ci\`encies
\\
08193 Bellaterra (Barcelona), Catalonia
}
\email{xtolsa@mat.uab.cat}

\thanks{M.M. was supported  by IKERBASQUE and partially supported by the grant IT-641-13 (Basque Government). X.T. was supported by the ERC grant 320501 of the European Research Council (FP7/2007-2013) and partially supported by MTM-2016-77635-P,  MDM-2014-044 (MICINN, Spain), 2017-SGR-395 (Catalonia), and by Marie Curie ITN MAnET (FP7-607647).
}

\begin{abstract}
Let $\Omega\subset\R^{n+1}$ be an open set with $n$-AD-regular boundary. In this paper we prove that if the harmonic measure for $\Omega$ satisfies the so-called weak-$A_\infty$ condition, then $\Omega$ satisfies a suitable connectivity condition, namely the weak local John condition.
Together with other previous results by Hofmann and Martell, this implies that
the weak-$A_\infty$ condition for harmonic measure holds if and only if $\partial\Omega$ is uniformly 
$n$-rectifiable and the weak local John condition is satisfied.
This yields the first geometric characterization of the weak-$A_\infty$ condition for harmonic measure,
which is important because of its connection with the Dirichlet problem for the Laplace equation.
\end{abstract}

\maketitle

% and rectifiability}
%{Some results and counterexamples regarding the $\beta_p$ coefficients and rectifiability}

%\author{Jonas Azzam}
%\address{Departament de Matem\`atiques\\ Universitat Aut\`onoma de Barcelona \\ Edifici C Facultat de Ci\`encies\\08193 Bellaterra (Barcelona, Catalonia) }
%\email{jazzam@mat.uab.cat}

\tableofcontents

\allowdisplaybreaks

\section{Introduction}

 The weak-$A_\infty$ condition for harmonic measure of an open set $\Omega \subset \R^{n+1}$ is a
quantitative version of absolute continuity of harmonic measure with respect to the surface measure.
In this paper we complete one of the fundamental steps for the characterization of the weak-$A_\infty$ condition for harmonic measure 
 in  terms of quantitative rectifiability of the boundary $\partial \Omega$ and a quantitative connectivity property of $\Omega$. More precisely, we show that if the weak-$A_\infty$ condition holds, then the so-called local John condition is satisfied. Together with previous results by Hofmann and Martell, this yields the aforementioned characterization.
 
The fact that rectifiability plays a fundamental role in the characterization of absolute continuity of harmonic measure with respect to surface measure has been well known since 1916 by the classical theorem
of F. and M. Riesz \cite{RR}.
 Recall that this
asserts that, given a simply connected domain $\Omega\subset\C$,  the rectifiability of $\partial\Omega$ implies that harmonic measure for $\Omega$ is absolutely continuous with respect to arc-length measure
 of the boundary. A local version of this theorem was obtained much later, in 1990, by Bishop and Jones \cite{BiJo}. For related results in higher dimensions see \cite{AAM}.
On the other hand, in the converse direction, it was shown recently in \cite{AHM3TV} that, for arbitrary
 open sets $\Omega\subset\R^{n+1}$, $n\geq1$, the  mutual absolute continuity of harmonic measure and surface measure (i.e.\ $n$-dimensional Hausdorff measure, which we will denote by $\HH^n$) in a subset $E\subset\partial\Omega$ implies  the $n$-rectifiability of $E$.

To describe other results of more quantitative nature we need now to introduce some notation and definitions.
A set $E\subset\R^{n+1}$ is called {\it $n$-AD-regular} if there exists some constant $C_0>0$ such that
$$C_0^{-1}r^n\leq\HH^n(E\cap B(x,r))\leq C_0\,r^n\quad \mbox{ for all $x\in E$ and $0<r\leq \diam(E)$.}$$
The set $E\subset\R^{n+1}$ is  {\it uniformly  $n$-rectifiable} if it is 
$n$-AD-regular and
there exist constants $\theta, M >0$ such that for all $x \in E$ and all $0<r\leq \diam(E)$ 
there is a Lipschitz mapping $g$ from the ball $B_n(0,r)$ in $\R^{n}$ to $\R^d$ with $\text{Lip}(g) \leq M$ such that
$$
\HH^n (E\cap B(x,r)\cap g(B_{n}(0,r)))\geq \theta r^{n}.$$
Uniform $n$-rectifiability is a quantitative version of $n$-rectifiability introduced by David and Semmes
(see \cite{DS1} and \cite{DS2}).

Let $\Omega\subset\R^{n+1}$ be open. One says that this satisfies the corkscrew condition if for
every $x\in\partial\Omega$ and $0<\rho\leq\diam(\Omega)$ there exists a ball $B\subset B(x,\rho)\cap\Omega$
with radius $r(B)\geq c\,\rho$, for some fixed $c>0$.

Given $p\in\Omega$, we denote by $\omega^p$ the harmonic measure for
$\Omega$ with pole at $p$.
Assume that $\partial\Omega$ has locally finite $\HH^n$-measure.   We say that the harmonic measure for $\Omega$ satisfies the {\it weak-$A_\infty$ condition} if for every 
$\ve_0 \in (0,1)$
 there exists $\delta_0 \in (0,1)$ such that for every ball $B$ centered at $\partial\Omega$ and all
 $p\in\Omega\setminus 4B$ the
 following holds: for any subset $E \subset B\cap\partial\Omega$,
 \begin{equation}\label{eq*fgh}
\mbox{if}\quad\HH^n(E)\leq \delta_0\,\HH^n(B\cap\partial\Omega), \quad\text{ then }\quad \omega^{p}(E)\leq \ve_0\,\omega^{p}(2B).
\end{equation}
In the case when the harmonic measure is doubling, that is, there is some constant $C>0$ such that
$$\omega^p(2B)\leq C\,\omega^p(B)\qquad\mbox{ for any ball $B$ centered at $\Omega$ and all $p\in\Omega$,}$$
the weak-$A_\infty$ condition coincides with the more familiar $A_\infty$ condition for $\omega^p$ (uniform on $p$).
Both the $A_\infty$ and weak-$A_\infty$ condition should be understood as quantitative versions of
the notion of absolute continuity. We will write $\omega\in A_\infty(\HH^1|_{\partial \Omega})$ and
$\omega\in\textup{weak-}A_\infty(\HH^1|_{\partial \Omega})$ to indicate that the harmonic measure
satisfies the $A_\infty$ and weak-$A_\infty$ conditions, respectively.

The weak-$A_\infty$ condition is particularly important from a PDE perspective. In fact, Hofmann and Le showed in \cite{HLe} that, 
if we assume $\Omega$ to satisfy the corkscrew condition and $\partial\Omega$ to
be $n$-AD-regular, then the Dirichlet problem
is BMO-solvable for the Laplace equation if and only if the harmonic measure is in weak-$A_\infty$.
So a geometric description of the domains $\Omega$ such that $\omega\in\textup{weak-}A_\infty$ is particularly desirable.

The first result of quantitative nature involving harmonic measure and rectifiability was obtained by
Lavrentiev \cite{Lav} in 1936 for planar domains. He showed that if $\Omega \subset \C$ is a simply connected  domain which is bounded by a chord-arc curve, then  $\omega \in A_\infty(\HH^1|_{\partial \Omega})$.
A fundamental result in arbitrary dimensions was obtained much later by Dahlberg \cite{Dahlberg}. He showed that if
$\Omega\subset\R^{n+1}$ is a bounded Lipschitz domain, then the harmonic measure satisfies the
reverse H\"older condition $B_2$ and thus it belongs to $A_\infty(\HH^1|_{\partial \Omega})$.
This result was extended to chord-arc domains by
 David and Jerison \cite{DJ}, and independently by Semmes \cite{Se}. They proved that chord-arc domains in 
 $\R^{n+1}$ (i.e., NTA domains with $n$-AD regular boundaries) have interior big pieces of Lipschitz, implying that $\omega \in A_\infty(\HH^n|_{\partial \Omega})$. 

In connection with harmonic measure, the weak-$A_\infty$ condition first appeared in the work by
 Bennewitz and Lewis in \cite{BL}, where it was shown that if the boundary of  
 $\Omega\subset \R^{n+1}$ is $n$-AD-regular and $\Omega$ has interior big pieces of Lipschitz domains, then $\omega \in \textup{weak-}A_\infty(\HH^n|_{\partial \Omega})$. They also showed that this is the best one can expect under these assumptions on the geometry of the domain. One can also show by the arugments in \cite{DJ} that this still holds if we replace Lipschitz with chord-arc subdomains. 
 
Later, Hofmann and Martell \cite{HM1}, and in collaboration with Uriarte-Tuero \cite{HMU}, showed that for a uniform domain with $n$-AD regular boundary, $\omega \in \textup{weak-}A_\infty(\HH^n|_{\partial \Omega})$ if and only if $\partial \Omega$ is uniformly $n$-rectifiable. This was further improved in \cite{AHMNT} where it was shown that any uniform domain with uniformly $n$-rectifiable boundary is in fact NTA and thus $\omega \in A_\infty(\HH^n|_{\partial \Omega})$. In \cite{HM2}\footnote{This result was published in \cite{HLMN}.} Hofmann and Martell removed the uniformity assumption entirely by showing that for a domain with $n$-AD-regular boundary that satisfies the corkscrew condition, if $\omega \in \textup{weak-}A_\infty(\HH^n|_{\partial \Omega})$, then $\partial \Omega$ is uniformly $n$-rectifiable. 
 This result was later extended to the case when the surface measure is non-doubling
in \cite{MT}.

Also note that according to Bishop and Jones' example in \cite{BiJo}, there exists an infinitely connected planar domain whose boundary is uniformly $1$-rectifiable but $\omega$ is not absolutely continuous  to arc-length. In fact, by \cite{GMT} and \cite{HMM}, the uniform rectifiability of $\partial
\Omega$ is equivalent to the existence of a suitable corona type decomposition of $\partial
\Omega$ in terms of harmonic measure (and also equivalent to a Carleson type condition for the gradient of bounded harmonic functions). 
So uniform rectifiability alone cannot characterize
the weak-$A_\infty$ condition for harmonic measure. 

 The first named author of the current manuscript recently showed in \cite{Azz2} that if a domain is semi-uniform and has uniformly rectifiable boundary, then harmonic measure is in $A_{\infty}$. Aikawa and Hirata had shown previously in \cite{AiHi} that a domain is semi-uniform if and only if the harmonic measure is doubling, which happens, in particular, if harmonic measure is in $A_{\infty}$ (they also assumed the domains were John but this assumption was removed in \cite{Azz2}). This and \cite{HM2} show that the $A_{\infty}$ condition implies semi-uniformity of the domain and uniform rectifiability of the boundary. Thus, the combination of these works yields a geometric characterization of the $A_{\infty}$ property.  

Hofmann and Martell, however, introduced an a priori weaker connectivity condition than interior big pieces of chord-arc domains that is sufficient for the weak-$A_{\infty}$ condition. 
Given $x\in \Omega$, $y\in\partial\Omega$, and $\lambda>0$, a {\it $\lambda$-carrot curve} (or just {\it carrot curve}) from $x$ to $y$ is a curve $\gamma\subset \Omega\cup\{y\}$ with end-points $x$ and $y$ such that $\delta_\Omega(z):=\dist(z,\partial\Omega)\geq \kappa\, \HH^{1}(\gamma(y,z))$ for all $z\in \gamma$, where $\gamma(y,z)$ is the arc in $\gamma$ between $y$ and $z$. 

One says that $\Omega$ satisfies the {\it weak local John condition} (with parameters $\lambda,\theta,\Lambda$) if
there are constants $\lambda,\theta\in (0,1)$ and $\Lambda \geq 2$ such that for every $x\in\Omega$
there is a Borel subset $F\subset B(x,\Lambda \delta_\Omega(x))\cap \partial \Omega)$ with
$\HH^n(F)\geq \theta\,\HH^n(B(x,\Lambda \delta_\Omega(x))\cap \partial \Omega)$
such that
every $y\in F$ can be joined to $x$ by a $\lambda$-carrot curve.
Note that the weak local John condition is weaker than semi-uniformity: rather than requiring nice carrot curves to every point on the boundary, there are only nice curves to points in a big piece.

In \cite{HM3} Hofmann and Martell showed that if $\Omega\subset\R^{n+1}$ is open (not necessarily connected), with
a uniformly rectifiable boundary, and $\Omega$ 
 satisfies the weak local John condition,
then harmonic measure is in weak-$A_\infty$. In the same work they conjectured that, 
conversely, if the harmonic measure is in weak-$A_\infty$, then the weak local John condition holds.

Our main result confirms this conjecture:

\begin{theorem}\label{teo1}
Let $\Omega\subset\R^{n+1}$, $n\geq2$, be an open set with $n$-AD-regular boundary. If the harmonic measure for $\Omega$ satisfies the
weak-$A_\infty$ condition, then $\Omega$ satisfies the weak local John condition.
\end{theorem}

After the publication of a first version of our paper in Arxiv, Hofmann and Martell also updated their paper \cite{HM3} to show that the weak local John condition implies interior big pieces of chord-arc domains. 
See \cite{HM3} for the precise definition of ``interior big pieces of chord-arc domains''.
Thus, combining our results with the main result of \cite{HM3}, we can conclude the following.

%{\color{blue} Hofmann and Martell have also updated their paper \cite{HM3} to show that the weak local John condition implies big pieces of chord-arc subdomains. Thus, combining our results with the main result of \cite{HM3}, we can conclude the following.}

\begin{coro}
Let $\Omega\subset\R^{n+1}$, $n\geq2$, be an open set with $n$-AD-regular boundary satisfying the corkscrew condition. The harmonic measure for
$\Omega$ is in weak-$A_\infty$ if and only if $\partial\Omega$ is uniformly $n$-rectifiable and $\Omega$ 
 satisfies the weak local John condition, if and only if $\Omega$ has interior big pieces of chord-arc domains. 
\end{coro}

Some of the difficulties that we have to overcome to prove Theorem \ref{teo1} arise from the fact
that the weak-$A_\infty$-condition does not imply any doubling condition on harmonic measure. 
Roughly speaking, given a ball $B$ centered at in $\partial\Omega$ and $x\in\Omega$, if $\omega^{x}_{\Omega}(B)$ is large, then $x$ should be well connected to a big piece of $\partial\Omega \cap B$  (though not necessarily any point in $B$). If we knew that the doubling property holds for each ball and also for different choices of $x$, then we would be able to piece together nice Harnack chains between different base points and the boundary. The weak $A_{\infty}$-condition, however, at best implies that $\omega_{\Omega}^{x}$ is doubling on balls centered on some large subset of the boundary, and this large subset may change as one changes the pole. So it is difficult to compare harmonic measure with respect to different poles in $\Omega$ (in fact, they 
may be mutually singular when $\Omega$ is not connected). 

Because of the reasons above, to prove Theorem \ref{teo1} we cannot use arguments similar to the ones in \cite{AiHi} or \cite{Azz2}. %"}
In fact,  we have to prove a local result which involves 
only one pole and one ball which has its own interest. See the Main Lemma \ref{lemhc} for more details. 
Two essential ingredients of the proof are a corona type decomposition (whose existence is ensured
by the uniform $n$-rectifiability of the boundary) and the Alt-Caffarelli-Friedman monotonicity formula
\cite{ACF}. This formula is used in some of the connectivity arguments in this paper. This allows to connect by carrot curves corkscrew points where the Green function is not too small to other corkscrews at a larger distance from the boundary
where the Green function is still not too small (see Lemma \ref{lemshortjumps} for the precise statement).
See also the work \cite{AGMT} for another related application of the Alt-Caffarelli-Friedman formula in connection with elliptic measure.

Two important steps of the proof of the Main Lemma \ref{lemhc} (and so of Theorem \ref{teo1}) are 
the Geometric Lemma \ref{lemgeom} and the Key Lemma \ref{keylemma}. 
An essential idea consists of distinguishing cubes with ``two well separated
big corkscrews'' (see Subsection \ref{subsepcork} for the precise definition).
In the Geometric Lemma \ref{lemhc}
we construct two disjoint open sets satisfying a John condition associated to trees involving this type of cubes, so that the boundaries of
the open sets are located in places where the Green function is very small. This construction
is only possible because the associated tree involves  only cubes with two well separated
big corkscrews. 
The existence of these cubes is an obstacle for the construction of carrot curves. However, in a sense,
in the Key Lemma \ref{keylemma} we take advantage of their existence to obtain some delicate estimates
for the Green function on some corkscrew points.

\vv

We would like to thank 
Jos\'{e} Mar\'{i}a Martell for several comments on a first a version of this paper.

\vv

% ***************************************************************************

\section{Preliminaries}

We will write $a\lesssim b$ if there is $C>0$ so that $a\leq Cb$ and $a\lesssim_{t} b$ if the constant $C$ depends on the parameter $t$. We write $a\approx b$ to mean $a\lesssim b\lesssim a$ and define $a\approx_{t}b$ similarly. Sometimes, given a measure $\nu$, we will also use the notation $\avint g\,d\nu$ for the average $\nu(F)^{-1}\int_{F}g\,d\nu$.
 
In the whole paper, $\Omega$ will be an open set in $\R^{n+1}$, with $n\geq2$.\vv

\subsection{The dyadic lattice $\DD_\mu$}\label{subsec:dyadic}

Given an $n$-AD-regular measure $\mu$ in $\R^{n+1}$ we consider 
the dyadic lattice of ``cubes'' built by David and Semmes in \cite[Chapter 3 of Part I]{DS2}. The properties satisfied by $\DD_\mu$ are the following. 
Assume first, for simplicity, that $\diam(\supp\mu)=\infty$). Then for each $j\in\Z$ there exists a family $\DD_{\mu,j}$ of Borel subsets of $\supp\mu$ (the dyadic cubes of the $j$-th generation) such that:
\begin{itemize}
\item[$(a)$] each $\DD_{\mu,j}$ is a partition of $\supp\mu$, i.e.\ $\supp\mu=\bigcup_{Q\in \DD_{\mu,j}} Q$ and $Q\cap Q'=\varnothing$ whenever $Q,Q'\in\DD_{\mu,j}$ and
$Q\neq Q'$;
\item[$(b)$] if $Q\in\DD_{\mu,j}$ and $Q'\in\DD_{\mu,k}$ with $k\leq j$, then either $Q\subset Q'$ or $Q\cap Q'=\varnothing$;
\item[$(c)$] for all $j\in\Z$ and $Q\in\DD_{\mu,j}$, we have $2^{-j}\lesssim\diam(Q)\leq2^{-j}$ and $\mu(Q)\approx 2^{-jn}$;
\item[$(d)$] there exists $C>0$ such that, for all $j\in\Z$, $Q\in\DD_{\mu,j}$, and $0<\tau<1$,
\begin{equation}\label{small boundary condition}
\begin{split}
\mu\big(\{x\in Q:\, &\dist(x,\supp\mu\setminus Q)\leq\tau2^{-j}\}\big)\\&+\mu\big(\{x\in \supp\mu\setminus Q:\, \dist(x,Q)\leq\tau2^{-j}\}\big)\leq C\tau^{1/C}2^{-jn}.
\end{split}
\end{equation}
This property is usually called the {\em small boundaries condition}.
From (\ref{small boundary condition}), it follows that there is a point $z_Q\in Q$ (the center of $Q$) such that $\dist(z_Q,\supp\mu\setminus Q)\gtrsim 2^{-j}$ (see \cite[Lemma 3.5 of Part I]{DS2}).
\end{itemize}
We set $\DD_\mu:=\bigcup_{j\in\Z}\DD_{\mu,j}$, and for $Q\in\DD_\mu$, we denote write $J(Q)=j$ if $Q\in\DD_{\mu,j}$. 

In case that $\diam(\supp\mu)<\infty$, the families $\DD_{\mu,j}$ are only defined for $j\geq j_0$, with
$2^{-j_0}\approx \diam(\supp\mu)$, and the same properties above hold for $\DD_\mu:=\bigcup_{j\geq j_0}\DD_{\mu,j}$.

Given a cube $Q\in\DD_{\mu,j}$, we say that its side length is $2^{-j}$, and we denote it by $\ell(Q)$. Notice that $\diam(Q)\leq\ell(Q)$. 
We also denote 
\begin{equation}\label{defbq}
B_Q:=B(z_Q,4\ell(Q)),
\end{equation}
and for $\lambda>1$, we write
$$\lambda Q = \bigl\{x\in \supp\mu:\, \dist(x,Q)\leq (\lambda-1)\,\ell(Q)\bigr\}.$$

Given $R\in\DD_\mu$, we set
$\DD_\mu(R):=\{Q\in\DD_\mu:Q\subset R\}$.
We also let $\DD_{\mu,j}(R)$ be
the family of cubes $Q\in\DD_\mu(R)$ such that $\ell(Q)= 2^{-j}\ell(R)$.
%, while $\DD_{\mu}^{j}(R)$ stands for the family of cubes $Q\in\DD_\mu(R)$ such that $\ell(Q)\geq 2^{-j}\ell(R).$.

\vv
% ***************************************************************************

\subsection{Uniform $n$-rectifiability}

A set $E\subset \R^{n+1}$ is called $n$-{\textit {rectifiable}} if there are Lipschitz maps
$f_i:\R^n\to\R^d$, $i=1,2,\ldots$, such that 
\begin{equation}\label{eq001}
\HH^n\biggl(E\setminus\bigcup_i f_i(\R^n)\biggr) = 0.
\end{equation} 
Recall that the notion of uniform $n$-rectifiability is a quantitative version of $n$-rectifiability. It is very easy to check that uniform $n$-rectifiability implies $n$-rectifiability.

Given a ball $B\subset \R^{n+1}$, we denote
\begin{equation}\label{defbbeta}
b\beta_E(B) = \inf_L \frac1{r(B)}\Bigl(\sup_{y\in E\cap B} \dist(y,L) + \sup_{y\in L\cap B}\dist(y,E)\Bigr),
\end{equation}
where the infimum is taken over all the affine $n$-planes that intersect $B$.
The following result is due to David and Semmes:

\begin{theorem}\label{teods}
Let $E\subset\R^{n+1}$ be $n$-AD-regular. Denote $\mu=\HH^n|_E$ and let $\DD_\mu$ be the associated dyadic lattice. Then, $E$ is uniformly $n$-rectifiable if and only if, for any $\ve>0$,
$$\sum_{\substack{Q\in\DD_\mu:Q\subset R,\\ b\beta(3B_Q)>\ve}} \mu(Q) \leq C(\ve)\,\mu(R)\quad \mbox{ for all $R\in\DD_\mu$.}$$
\end{theorem}

The constant $3$ multiplying $B_Q$ in the estimate above can be replaced by any number larger than $1$. For the proof, see
\cite[Chapter II-2]{DS2}.

Recall also the following result (see \cite{HLMN} or \cite{MT}).

\begin{theorem}\label{teo*}
Let $\Omega\subset\R^{n+1}$ be an open set with $n$-AD-regular boundary such that the harmonic measure in $\Omega$ belongs to
weak-$A_\infty$. Then $\partial\Omega$ is uniformly $n$-rectifiable.
\end{theorem}
\vv
% ***************************************************************************

\subsection{Harmonic measure}
From now on we assume that $\Omega\subset\R^{n+1}$ is an open set with $n$-AD-regular boundary such that the harmonic measure in $\Omega$ belongs to
weak $A_\infty$. We denote by $\mu$ the surface measure in $\partial\Omega$. That is, $\mu = \HH^n|_{\partial\Omega}$.
We also consider the dyadic lattice $\DD_\mu$ associated with $\mu$.
The AD-regularity constant of $\partial\Omega$ is denoted by $C_0$.

We denote by $\omega^p$ the harmonic measure with pole at $p$ of $\Omega$, and by $g(\cdot,\cdot)$ the Green function. We write $\delta_\Omega(x) = \dist(x,\partial\Omega)$.

The following well known result is sometimes called ``Bourgain's estimate":

\begin{lemma}\label{l:bourgain}
Let $\Omega\subsetneq \R^{n+1}$ be open with $n$-AD-regular boundary,  $x\in \partial\Omega$, and $0<r\leq\diam(\partial\Omega)/2$.  Then 
\begin{equation}\label{e:bourgain}
\omega^y(B(x,2r))\geq c >0, \;\; \mbox{ for all }y\in \Omega\cap \overline B(x,r)
\end{equation}
where $c$ depends on $n$ and the $n$-AD-regularity constant of $\partial\Omega$.
\end{lemma}

The following is also well known.

\begin{lemma}\label{lem1}
Let $p,q\in\Omega$ be such $|p-q|\geq 4\,\delta_\Omega(q)$.
Then,
$$g(p,q)\leq C\,\frac{\omega^p(B(q,4\delta_\Omega(q)))}{\delta_\Omega(q)^{n-1}}.$$
\end{lemma}

The following lemma is also known. See \cite[Lemma 3.14]{HLMN}, for example.

\begin{lemma}\label{lem1'}
Let $\Omega\subsetneq \R^{n+1}$ be open with $n$-AD-regular boundary and let $p\in\Omega$. Let $B$ be a ball centered at 
$\partial\Omega$ such that $p\not\in 8B$. Then
$$\avint_B g(p,x)\,dx \leq C\,\frac{\omega^p(4B)}{r(B)^{n-1}}.$$
\end{lemma}

\vv

\begin{lemma}\label{lem333}
Let $\Omega\subsetneq\R^{n+1}$ be open with $n$-AD-regular boundary. Let $x\in\partial\Omega$ and $0<r<\diam(\Omega)$.
Let $u$ be a non-negative harmonic function in $B(x,4r)\cap \Omega$ and continuous in $B(x,4r)\cap \overline\Omega$
such that $u\equiv 0$ in $\partial\Omega\cap B(x,4r)$. Then extending $u$ by $0$ in $B(x,4r)\setminus \overline\Omega$,
there exists a constant $\alpha>0$ such that, for all $y,z\in B(x,r)$,
$$|u(y)-u(z)|\leq C\,\left(\frac{|y-z|}r\right)^\alpha \!\sup_{B(x,2r)}u
\leq C\,\left(\frac{|y-z|}r\right)^\alpha \;\avint_{B(x,4r)}u,
%\quad \mbox{for all $y,z\in B(x,r)$,}
$$
where $C$ and $\alpha$ depend on $n$ and the AD-regularity of $\partial \Omega$.
In particular,
$$u(y)\leq C\,\left(\frac{\delta_\Omega(y)}r\right)^\alpha \!\sup_{B(x,2r)}u
\leq C\,\left(\frac{\delta_\Omega(y)}r\right)^\alpha \;\avint_{B(x,4r)}u.$$
\end{lemma}

\vv
The next result provides a partial converse to Lemma \ref{l:bourgain}

\begin{lemma}\label{lem2}
Let $\Omega\subsetneq\R^{n+1}$ be open with $n$-AD-regular boundary. 
Let $p\in\Omega$ and let $Q\in\DD_\mu$ be such that $p\not\in 2Q$. 
Suppose that $\omega^p(Q)\approx \omega^p(2Q)$.
Then there exists some $q\in\Omega$ such that
$$\ell(Q)\lesssim \delta_\Omega(q)\approx \dist(q,Q)\leq 4\diam(Q)$$
and 
$$\frac{\omega^p(2Q)}{\ell(Q)^{n-1}}\leq c\,g(p,q).$$
\end{lemma}

\begin{proof}
For a given $k_0\geq 2$ to be fixed below, let $P\in\DD_\mu$ be a cube contained in $Q$ with $\ell(P)=
2^{-k_0}\ell(Q)$ such that 
$$\omega_p(P)\approx_{k_0}\omega^p(Q).$$
 Let $\vphi_P$ be a $C^\infty$ function supported in $B_P$ which equals $1$ on $P$ and such that
 $\|\nabla\vphi_P\|_\infty\lesssim 1/\ell(P)$.
Then, choosing $k_0$ small enough so that $p\not\in 50B_P$, say, and applying Caccioppoli's inequality,
\begin{align*}
\omega^p(2Q)&\approx_{k_0} \omega^p(P) \leq \int \vphi_P\,d\omega^p = -\int\nabla_y g(p,y)\,\nabla\vphi_P(y)\,dy\\
& \lesssim \frac1{\ell(P)}\int_{B_P} |\nabla_y g(p,y)|\,dy \lesssim \ell(P)^n
\left(\;\avint_{B_P} |\nabla_y g(p,y)|^2\,dy \right)^{1/2}\\
&\lesssim \ell(P)^{n-1}
\left(\;\avint_{2B_P} |g(p,y)|^2\,dy \right)^{1/2}\lesssim \ell(P)^{n-1}
\;\avint_{3B_P} g(p,y)\,dy.
 \end{align*}
Applying  now Lemmas \ref{lem333} and \ref{lem1'} and taking $k_0$ small enough so that $24B_P\cap\partial\Omega\subset2Q$,  for any $a\in (0,1)$ we get
$$\;\avint_{y\in3B_P:\delta_\Omega(y)\leq a\ell(P)} g(p,y)\,dy \lesssim a^\alpha \;\avint_{6B_P} g(p,y)\,dy
\lesssim a^\alpha\,\frac{\omega^p(24B_P)}{\ell(P)^{n-1}}\lesssim a^\alpha\,\frac{\omega^p(2Q)}{\ell(P)^{n-1}}.
$$
From the estimates above we infer that
$$\omega^p(2Q)\lesssim_{k_0}	\ell(P)^{n-1}\;
\avint_{y\in3B_P:\delta_\Omega(y)\geq a\ell(P)} g(p,y)\,dy + a^\alpha\,\omega^p(2Q).$$
Hence, for $a$ small enough, we derive
$$\omega^p(2Q)\lesssim_{k_0}	\ell(P)^{n-1}\;
\avint_{y\in3B_P:\delta_\Omega(y)\geq a\ell(P)} g(p,y)\,dy,$$
which implies the existence of the point $q$ required in the lemma.
\end{proof}

\vv

\subsection{Harnack chains and carrots}

It will be more convenient for us to work with Harnack chains instead of curves. The existence of a carrot curve is equivalent to having what we call a good chain between points.

Let $x\in \Omega$, $y\in\overline\Omega$ be such that $\delta_\Omega(y)\leq \delta_\Omega(x)$, and let $C>1$.
A {\it $C$-good chain} (or $C$-good Harnack chain) from $x$ to $y$ is a sequence of balls $B_{1},B_{2},...$ (finite or infinite)  contained in $\Omega$ such that 
$x\in B_1$ and either
\begin{itemize}
\item $\lim_{j\to\infty} \dist(y,B_j)=0$ if $y\in \partial\Omega$, or
\item $y\in B_N$ if $y\in \Omega$, where $N$ is the number of elements of the sequence if this is finite,
\end{itemize}
and moreover the following holds:
\begin{itemize}
\item $B_j\cap B_{j+1}\neq \varnothing$ for all $j$, 
\item $C^{-1}\,\dist(B_j,\partial\Omega) \leq r(B_j)\leq C\,\dist(B_j,\partial\Omega)$ for all $j$,
\item $r(B_j)\leq C\,r(B_i)$ if $j>i$,
\item for each $t>0$ there are at most $C$ balls $B_j$ such that $t< r(B_j)\leq2t$.
\end{itemize}
Abusing language, sometimes we will omit the constant $C$ and we will just say ``good chain'' or ``good Harnack chain''.

\vv
Observe that in the definitions of carrot curves and good chains, the order of $x$ and $y$ is important: having a carrot curve from $x$ to $y$ is not equivalent to having one from $y$ to $x$, and similarly with good chains. 

\begin{lemma}\label{l:johntochain}
There is a carrot curve from $x\in \Omega$ to $y\in \overline{\Omega}$ if and only if there is a good Harnack chain from $x$ to $y$. 
\end{lemma}

%So when we prove the weak local John condition, we will do so by showing that, for every ball $B$ centered on $\partial\Omega$, there is $F\subseteq B\cap \partial\Omega$ so that, for any corkscrew point $x\in B$ and every $\xi\in F$, there is a good Harnack chain in $\Omega$ from $x$ converging to $\xi$. 

\begin{proof}
Let $\gamma$ be a carrot curve from $x$ to $y$.  We can assume $y\in \Omega$, since if $y\in \partial\Omega$, we can obtain this case by taking a limit of points $y_{j}\in \Omega$ converging to $y$. Let $\{B_{j}\}_{j=1}^{N}$ be a Vitali subcovering  of 
the family 
$\{B(z,\delta_{\Omega}(z)/10): z\in \gamma\}$ and let $r_{B_j}$ stand for the radius  and $x_{B_j}$ for the center of $B_j$. So the balls $B_j$ are disjoint and $3B_{j}$ cover $\gamma$. Note that for $t>0$, if $t< r_{B_{j}}\leq 2t$, 
\[
|x_{B_{j}}-y|
\leq \HH^1(\gamma(x_{B_{j}},y))
\lesssim \delta_\Omega (x_{B_{j}})
\approx r_{B_{j}}\leq 2t.
\]
In particular, since the $B_{j}$'s are disjoint, by volume considerations, there can only be boundedly many $B_{j}$ of radius between $t/2$ and $t$, say. Moreover, 
 we may order the balls $B_j$ so that $x\in 5B_1$ and 
$B_{j+1}$ is a ball $B_k$ such that $5B_k\cap 5B_j\neq\varnothing$
and $5\overline{B_k}$ contains the point from $\gamma\cap \bigcup_{h:5B_h\cap 5B_j\neq\varnothing}5\overline B_h$ which is maximal
in the natural order induced by $\gamma$ (so that $x$ is the minimal point in $\gamma$).
Then for $j>i$,
$$r_{B_{j}} \approx \delta_\Omega(x_{B_j}) \leq   |x_{B_j}- x_{B_i}|+ \delta_\Omega(x_{B_i})
\leq 
\HH^1(\gamma(x_{B_{i}},y)) + \delta_\Omega(x_{B_i}) 
\lesssim r_{B_{i}}.
$$
%Finally, note if $y'$ is the closest point to $y$ belonging to $\partial\Omega$, then for each $j\neq N$ (where $N$ is the number of balls $B_j$ if the sequence is finite),
%$$|x_{B_j}-y'| 
%\lesssim |x_{B_j}-y| + |y-y'| \lesssim  |x_{B_j}-y|+ r(B_N)\lesssim  \HH^1(\gamma(x_{B_j},y)) \lesssim \delta_\Omega(B_j) \lesssim |x_{B_j}-y'|. $$
%Thus there are only boundedly many $j$'s such that $|x_{B_{j}}-y'|\approx t$. Integrating over $t$ from $\delta_\Omega (y)$ to $C\delta_\Omega (x)$, we get that the total number of balls $B_{j}$ is at most a constant times $\log\frac{\delta_\Omega (x)}{\delta_\Omega (y)}+1$. 
This implies $5B_1,5B_2,\ldots$ is a $C$-good chain for a sufficiently big $C$.

Now suppose that we can find a good chain from $x$ to $y$, call it $B_{1},...,B_{N}$. Let $\gamma$ be the path obtained by connecting their centers in order. Let $z\in \gamma$. Then there is a $j$ such that $z\in [x_{B_{j}},x_{B_{j+1}}]$. Since $\{B_{i}\}_i$ is a good chain,\[
\HH^1(\gamma(z,y))
\leq |z-x_{B_{j+1}}| + \HH^1(\gamma(x_{B_{j+1}},y))
\leq r_{B_{j+1}}+\sum_{i=j}^{N} 2r_{B_{i}}
\lesssim r_{B_{j}}\approx \delta_\Omega(z). 
\]
Thus, $\gamma$ is a carrot curve from $x$ to $y$. 
\end{proof}

\vv

\subsection{The Alt-Caffarelli-Friedman formula}

\begin{theorem} \label{t:ACF} Let $B(x,R)\subset \R^{n+1}$, and let $u_1,u_2\in
W^{1,2}(B(x,R))\cap C(B(x,R))$ be nonnegative subharmonic functions. Suppose that $u_1(x)=u_2(x)=0$ and
that $u_1\cdot u_2\equiv 0$. Set
\begin{equation*}%\label{eq*123}
J_i(x,r) = \frac{1}{r^{2}} \int_{B(x,r)} \frac{|\nabla u_i(y)|^{2}}{|y-x|^{n-1}}dy,
\end{equation*}
and
\begin{equation}\label{e:Jxr}
J(x,r) = J_1(x,r)\,J_2(x,r).
\end{equation}
Then $J(x,r)$ is a non-decreasing function of $r\in (0,R)$ and $J(x,r)<\infty$ for all $r\in (0,R)$. That is,
\begin{equation}\label{e:gamma}
J(x,r_{1})\leq J(x,r_{2})<\infty \;\;  \mbox{ for } \;\; 0<r_{1}\leq r_{2}<R.
\end{equation}
Further, 
\begin{equation}\label{eqjr}
J_i(x,r)\lesssim \frac1{r^2}\,\|u_i\|_{\infty,B(x,2r)}^2.
\end{equation}
\end{theorem}

In the case of equality we have the following result (see \cite[Theorem 2.9]{PSU}).

\begin{theorem}\label{teoACFequ}
Let $B(x,R)$ and $u_1,u_2$ be as in Theorem \ref{t:ACF}. Suppose that $J(x,r_a)=J(x,r_b)$ for some $0<r_a<r_b<R$.
Then either one or the other of the following holds:
\begin{itemize}
\item[(a)] $u_1=0$ in $B(x,r_b)$ or $u_2=0$ in $B(x,r_b)$;

\item[(b)] there exists a unit vector $e$ and constants $k_1,k_2>0$ such that
$$u_1(y)= k_1\,((y-x)\cdot e)^+,\qquad u_2(y)= k_2\,((y-x)\cdot e)^-, \qquad\text{in $B(x,r_b)$}.$$
\end{itemize}
\end{theorem}

 We will also need the following auxiliary lemma.
\vv
\begin{lemma} \label{lem:ACF1} Let $B(x,R)\subset \R^{n+1}$, and let $\{u_i\}_{i\geq 1}
\subset W^{1,2}(B(x,R))\cap C(B(x,R))$ a sequence of
functions which are nonnegative, subharmonic, such that each $u_i$ is
harmonic in $\{y\in B(x,R):u_i(y)>0\}$ and $u_i(x)=0$. Suppose also that
$$\|u_i\|_{\infty,B(x,R)}\leq C_1\,R\quad\text{ and }\quad \|u_i\|_{{\rm Lip}^\alpha,B(x,R)}\leq C_1\,R^{1-\alpha}$$
for all $i\geq 1$.
Then, for every $0<r<R$  there exists a subsequence $\{u_{i_k}\}_{k\geq1}$ which converges uniformly in
$B(x,r)$ and weakly in $W^{1,2}(B(x,r))$ to some function $u\in W^{1,2}(B(x,r))\cap C(B(x,r))$, and moreover,
\begin{equation}\label{eqlim4}
\lim_{k\to\infty}  \int_{B(x,r)} \frac{|\nabla u_{i_k}(y)|^{2}}{|y-x|^{n-1}}dy =
 \int_{B(x,r)} \frac{|\nabla u(y)|^{2}}{|y-x|^{n-1}}dy.
 \end{equation}
\end{lemma}

\begin{proof}
The existence of a subsequence $\{u_{i_k}\}_{k\geq1}$ converging weakly in $W^{1,2}(B(x,r))$ and uniformly in
$B(x,r)$ to some function $u\in W^{1,2}(B(x,r))\cap C(B(x,r))$
is an immediate consequence of the Arzel\`a-Ascoli and the Banach-Alaoglu theorems.
Quite likely, the identity \rf{eqlim4} is also well known. However, for completeness, we will show the details.

Consider a non-negative subharmonic function $v\in W^{1,2}(B(x,R))\cap C(B(x,R))$ which is harmonic in $\{y\in B(x,R):v(y)>0\}$ so that $v(x)=0$. 
For $0<r<R$ and $0<\delta<R-r$, let $\vphi$ be a radial $C^\infty$ function such that $\chi_{B(x,r)} \leq  \vphi\leq
\chi_{B(x,r+\delta)}$. Let $\EE(y)=c_n^{-1}\,|y|^{1-n}$ be the fundamental solution of the Laplacian.
For $\ve>0$, denote $v_\ve=\max(v,\ve)-\ve$.
Then we have
\begin{align*}
\int \frac{|\nabla v_\ve(y)|^{2}}{|y-x|^{n-1}}\,\vphi(y)\,dy &= c_n\int \nabla v_\ve(y)\,\nabla(\EE(x-\cdot)\,v_\ve\,\vphi)(y) \,dy\\
&\quad - c_n\int \nabla v_\ve(y)\,\EE(x-y)\,v_\ve(y)\,\nabla\vphi(y) \,dy\\
& \quad - c_n\int \nabla v_\ve(y)\,\nabla_y \EE(x-y)\,v_\ve(y)\,\vphi(y) \,dy = c_n( I_1 - I_2 - I_3).
\end{align*}
Using the fact that $v_\ve$ is harmonic in $\{v_\ve>0\}$ and that $\EE(x-\cdot)\,v_\ve\,\vphi \in W^{1,2}_0(\{v_\ve>0\} \cap B(x,R))$ since  $\vphi$ is compactly supported in $B(x,R)$, $v_\ve=0$ on $\partial \{v_\ve>0\}$, and $x$ is far away from $\overline{\{v_\ve>0\}}$,
it follows easily that $I_1=0$. On the other hand, we have
\begin{align*}
2\,I_3 & = \int \nabla(v_\ve^2\,\vphi)(y)\,\nabla_y \EE(x-y)\,dy -
\int v_\ve(y)^2\,\nabla_y \EE(x-y)\,\nabla\vphi(y) \,dy \\
& = - v_\ve(x)^2-
\int v_\ve(y)^2\,\nabla_y \EE(x-y)\,\nabla\vphi(y) \,dy.
\end{align*}
Thus,
\begin{align*}
\int \frac{|\nabla v_\ve(y)|^{2}}{|y-x|^{n-1}}\,\vphi(y)\,dy & = - c_n\int \nabla v_\ve(y)\,\EE(x-y)\,v_\ve(y)\,\nabla\vphi(y) \,dy\\
&\quad - \frac{c_n}2\int v_\ve(y)^2\,\nabla_y \EE(x-y)\,\nabla\vphi(y) \,dy.
\end{align*}
Taking into account that $\supp \nabla\vphi$ is far away from $x$, letting $\ve\to0$, we obtain
\begin{align*}
\int \frac{|\nabla v(y)|^{2}}{|y-x|^{n-1}}\,\vphi(y)\,dy & = - c_n\int \nabla v(y)\,\EE(x-y)\,v(y)\,\nabla\vphi(y) \,dy\\
&\quad - \frac{c_n}2\int v(y)^2\,\nabla_y \EE(x-y)\,\nabla\vphi(y) \,dy.
\end{align*}

Using the preceding identity, it follows easily that
$$\lim_{k\to\infty}\int \frac{|\nabla u_{i_k}(y)|^{2}}{|y-x|^{n-1}}\,\vphi(y)\,dy = \int \frac{|\nabla u(y)|^{2}}{|y-x|^{n-1}}\,\vphi(y)\,dy.$$
Indeed, $\lim_{k\to \infty}u_{i_k}(x)^2 = u(x)^2$. Also, it is clear that 
$$\lim_{k\to \infty} \int u_{i_k}(y)^2\,\nabla_y \EE(x-y)\,\nabla\vphi(y) \,dy = \int u(y)^2\,\nabla_y \EE(x-y)\,\nabla\vphi(y) \,dy.$$
Further,
\begin{align*}
\int \nabla u_{i_k}(y)\,\EE(x-y)\,u_{i_k}(y)\,\nabla\vphi(y) \,dy & = 
\int \nabla u_{i_k}(y)\,\EE(x-y)\,u(y)\,\nabla\vphi(y) \,dy \\
&\quad + \int \nabla u_{i_k}(y)\,\EE(x-y)\,(u_{i_k}(y)-u(y))\,\nabla\vphi(y) \,dy\\
&\stackrel{k\to\infty}\to \int \nabla u(y)\,\EE(x-y)\,u(y)\,\nabla\vphi(y) \,dy,
\end{align*}
by the weak convergence of $u_{i_k}$ in $W^{1,2}(B(x,R))$ and the uniform convergence in $B(x,r+\delta)$,
since $\supp\nabla\vphi$ is far away from $x$.

Let $\psi$ be a radial $C^\infty$ function such that $\chi_{B(x,r-\delta)} \leq  \psi\leq
\chi_{B(x,r)}$. The same argument as above shows that
$$\lim_{k\to\infty}\int \frac{|\nabla u_{i_k}(y)|^{2}}{|y-x|^{n-1}}\,\psi(y)\,dy = \int \frac{|\nabla u(y)|^{2}}{|y-x|^{n-1}}\,\psi(y)\,dy.$$
Consequently,
$$\limsup_{k\to\infty}   \int_{B(x,r)} \frac{|\nabla u_{i_k}(y)|^{2}}{|y-x|^{n-1}}dy \leq 
\lim_{k\to\infty}\int \frac{|\nabla u_{i_k}(y)|^{2}}{|y-x|^{n-1}}\,\vphi(y)\,dy = \int \frac{|\nabla u(y)|^{2}}{|y-x|^{n-1}}\,\vphi(y)\,dy,$$
and also
$$\liminf_{k\to\infty}   \int_{B(x,r)} \frac{|\nabla u_{i_k}(y)|^{2}}{|y-x|^{n-1}}dy \geq 
\lim_{k\to\infty}\int \frac{|\nabla u_{i_k}(y)|^{2}}{|y-x|^{n-1}}\,\psi(y)\,dy = \int \frac{|\nabla u(y)|^{2}}{|y-x|^{n-1}}\,\psi(y)\,dy.$$
Since $\delta>0$ can be taken arbitrarily small, \rf{eqlim4} follows.
\end{proof}

\vv

\begin{lemma} \label{lem:ACF2} Let $B(x,2R)\subset \R^{n+1}$, and let $u_1,u_2\in
W^{1,2}(B(x,2R))\cap C(B(x,2R))$ be nonnegative subharmonic functions such that each $u_i$ is harmonic
in $\{y\in B(x,2R):u_i(y)>0\}$.
 Suppose that $u_1(x)=u_2(x)=0$ and
that $u_1\cdot u_2\equiv 0$. 
Assume also that
$$\|u_i\|_{\infty,B(x,2R)}\leq C_1\,R\quad\text{ and }\quad \|u_i\|_{{\rm Lip}^\alpha,B(x,2R)}\leq C_1\,R^{1-\alpha}\quad \mbox{ for $i=1,2$.}$$
For any $\ve>0$, there exists some $\delta>0$ such that if
$$J(x,R)\leq (1+\delta)\, J(x,\tfrac12R),$$
with $J(\cdot,\cdot)$ defined in \rf{e:Jxr},
then either
one or the other of the following holds:
\begin{itemize}
\item[(a)] $\|u_1\|_{\infty,B(x,R)}\leq \ve\,R$ or $\|u_2\|_{\infty,B(x,R)}\leq \ve\,R$;

\item[(b)] there exists a unit vector $e$ and constants $k_1,k_2>0$ such that
$$\|u_1- k_1\,((\cdot-x)\cdot e)^+\|_{\infty,B(x,R)}\leq \ve\,R
,\qquad \|u_2-k_2\,((\cdot-x)\cdot e)^-\|_{\infty,B(x,R)}\leq \ve\,R 
.$$
\end{itemize}
The constant $\delta$ depends only on $n,\alpha,C_1,\ve$.
\end{lemma}

\begin{proof}
Suppose that the conclusion of the lemma fails.
By replacing $u_i(y)$ by $\frac1{R}\,u_i(R(y+x))$, we  can assume that $x=0$ and $R=1$.
Let $\ve>0$, and for each $\delta=1/k$ and $i=1,2$, consider functions $u_{i,k}$ satisfying the
assumptions of the lemma and such that neither (a) nor (b) holds for them.
By Lemma \ref{lem:ACF1}, there exist subsequences (which we still denote by $\{u_{i,k}\}_k$) 
which converge uniformly in
$B(0,\frac32)$ and weakly in $W^{1,2}(B(0,\frac32))$ to some functions $u_i\in W^{1,2}(B(0,\frac32))\cap C(B(0,\frac32))$, and moreover,
$$\lim_{k\to\infty}  \int_{B(0,r)} \frac{|\nabla u_{i,k}(y)|^{2}}{|y|^{n-1}}dy =
 \int_{B(0,r)} \frac{|\nabla u_i(y)|^{2}}{|y|^{n-1}}dy$$
both for $r=1$ and $r=1/2$. Clearly, the functions $u_i$ are non-negative, subharmonic, and $u_1\cdot u_2=0$. Hence, by Theorem \ref{teoACFequ}, one of the following holds:
\begin{itemize}
\item[(a')] $u_1=0$ in $B(0,1)$ or $u_2=0$ in $B(0,1)$;

\item[(b')] there exists a unit vector $e$ and constants $k_1,k_2>0$ such that
$$u_1(y)= k_1\,(y\cdot e)^+,\qquad u_2(y)= k_2\,(y\cdot e)^-, \qquad\text{in $B(0,1)$}.$$
\end{itemize}
However, the fact that neither (a) nor (b) holds for any pair $u_{1,k}$, $u_{2,k}$, together with the uniform
convergence of $\{u_{i,k}\}_k$, implies that neither (a') nor (b') can hold, and thus we get a contradiction.
\end{proof}

\subsection{The Main Lemma} Let $B\subset\R^{n+1}$ be a ball centerer at $\partial\Omega$ and let $p
\in\Omega$.
 We say that $\omega^p$ satisfies the {\it weak-$A_\infty$ condition in $B$} if for every 
$\ve_0 \in (0,1)$
 there exists $\delta_0 \in (0,1)$ such that the
 following holds: for any subset $E \subset B\cap\partial\Omega$,
 \begin{equation*}
\mbox{if}\quad\HH^n(E)\leq \delta_0\,\HH^n(B\cap\partial\Omega), \quad\text{ then }\quad \omega^{p}(E)\leq \ve_0\,\omega^{p}(2B).
\end{equation*}

In the next sections we will prove the following.

\begin{mlemma}\label{lemhc} 
Let $\Omega\subset\R^{n+1}$ have $n$-AD-regular boundary.
Let $R_0\in\DD_\mu$ and let $p\in\Omega\setminus 4B_{R_0}$ be a point such that 
$$c\,\ell(R_0)\leq \dist(p,\partial\Omega)\leq \dist(p,R_0)\leq c^{-1}\,\ell(R_0)$$
and $\omega^{p}(R_0)\geq c'>0$. Suppose that
$\omega^p$ satisfies the weak-$A_\infty$ condition in $B_{R_0}$.
Then there exists a subset $\Con(R_0)\subset R_0$ and a constant $c''>0$ with
$\mu(\Con(R_0))\geq c''\,\mu(R_0)$ such that each point $x\in \Con(R_0)$ can be joined to $p$ by a carrot curve.
The constant $c''$ and the constants involved in the carrot condition only depend on $c,c',n$, the weak-$A_\infty$ condition, and the $n$-AD-regularity of $\mu$.
\end{mlemma}

The notation $\Con(\cdot)$ stands for ``connectable''.

It is easy to check that Theorem \ref{teo1} follows from this result. Indeed, given any $x\in\Omega$, we take a point $\xi\in\partial\Omega$ such that $|x-\xi|=\delta_\Omega(x)$. Then we consider the point $p$ in the segment  $[x,\xi]$ such 
that $|p-\xi|=\frac1{16}\,\delta_\Omega(x)$. By Lemma \ref{e:bourgain}, we have
$$\omega^p(B(\xi,\tfrac18\delta_\Omega(x)))\gtrsim 1,$$
because $p\in \frac12B(\xi,\tfrac18\delta_\Omega(x))$. Hence, by covering $B(\xi,\tfrac18\delta_\Omega(x))\cap\Omega$ with cubes $R\in\DD_\mu$ contained in $B(\xi,\tfrac14\delta_\Omega(x))\cap\partial\Omega$ with side length comparable to $\delta_\Omega(x)$
we deduce that at least one these cubes, call it $R_0$, satisfies $\omega^p(R_0)\gtrsim 1$. Further, by taking the side length small enough,
we may also assume that $p\not\in 4B_{R_0}$. So by applying Lemma \ref{lemhc} above we infer that there exists a subset $F:=\Con(R_0)\subset R_0$ with
$\mu(F)\geq c'\,\mu(R_0)\gtrsim \delta_\Omega(x)^n$ such that all $y\in F$ can be joined to $x$ by
a carrot curve, which proves that $\Omega$
satisfies the weak local John condition and concludes the proof of Theorem \ref{teo1}.
\vv

For simplicity, in the next sections {\bf we will assume that $\Omega=\R^{n+1}\setminus \partial\Omega$}.
At the end of the paper we will sketch the necessary changes for the general case.

\vv

% ***************************************************************************

% ***************************************************************************
\section{Short paths}\label{secshortp}

Let $p\in\Omega$ and $\Lambda>1$.
For $x\in\partial\Omega$, we write \hyperlink{WA-link}{$x\in \WA(p,\Lambda)$} if 
\begin{itemize}
\item $x\in B(p,10\delta_\Omega(p))\cap \partial\Omega$, and
\item for all $0<r\leq \delta_\Omega(p)$,
$$\Lambda^{-1}\,\frac{\mu(B(x,r))}{\mu(B(x,\delta_\Omega(p)))}\leq
\omega^p(B(x,r))\leq \Lambda\,\frac{\mu(B(x,r))}{\mu(B(x,\delta_\Omega(p)))}.$$
\end{itemize}
We will see in Section \ref{sechdld} that, under the assumptions of the Main Lemma \ref{lemhc},
 for some $\Lambda$ big enough,
\begin{equation}\label{eqwap}
\mu(\WA(p,\Lambda)\cap R_0)\gtrsim \mu(R_0).
\end{equation}

\vv

\begin{lemma}\label{lemfacc}
Let $p\in\Omega$, $x_0\in\WA(p,\Lambda)$, and $r\in (0,\delta_\Omega(p))$. Then there exists
$q\in B(x_0,r)$ such that, for some constant $\kappa\in (0,1/10)$,
\begin{itemize}
\item[(a)] $\delta_\Omega(q) \geq \kappa \,r$, and
\item [(b)]
$$\kappa\,\frac{\omega^p(B(x_0,r))}{r^{n-1}} \leq  g(p,q)\leq \kappa^{-1}\,\frac{\omega^p(B(x_0,r))}{r^{n-1}}
.$$
\end{itemize}
The constant $\kappa$ depends only on $\Lambda$, $n$, and $C_0$, the AD-regularity constant of $\partial\Omega$.
\end{lemma}

\begin{proof}
This follows easily from Lemmas \ref{lem1} and \ref{lem2}.
\end{proof}

\vv

\begin{lemma}[Short paths]\label{lemshortjumps}
Let $p\in\Omega$, $x_0\in\WA(p,\Lambda)$,  and for $0<r_0\leq \delta_\Omega(p)/4$, $0<\tau_0,\lambda_0 \leq1$, let $q\in\Omega$ be such that
\begin{equation}\label{eqass11}
q\in B(x_0,r_0),\quad \delta_\Omega(q)\geq \tau_0\,r_0,\quad g(p,q)\geq \lambda_0\,\frac{\delta_\Omega(q)}{\delta_\Omega(p)^n}.
\end{equation}
Then there exist constants $A_1>1$ and $0<a_1,\lambda_1<1$ such that for every $r\in (r_0,\delta_\Omega(p)/2)$,  there exists some point $q'\in\Omega$ such that
\begin{equation}\label{eqconc11}
q'\in B(x_0,A_1r),\quad \delta_\Omega(q')\geq \kappa\,|x_0-q'|\geq \kappa\,r,\quad g(p,q')\geq \lambda_1\,\frac{\delta_\Omega(q')}{\delta_\Omega(p)^n},
\end{equation}
and such that
$q$ and $q'$ can be joined by a curve $\gamma$ such that
$$\gamma\subset\{y\in B(x_0,A_1r):\dist(y,\partial\Omega)
>a_1\, r_0\}.$$
The parameters $\lambda_1,A_1,a_1$ depend only 
on $C_0,\Lambda,\lambda_0,\tau_0$ and the ratio $r/r_0$.
\end{lemma}

\begin{proof}
All the parameters in the lemma will be fixed along the proof. We assume that $A_1\gg \kappa^{-1}>1$.
First note that we may assume that $r< 2A_1^{-1} |x_0-p|$. Otherwise, we just take a point
$q'\in\Omega$ such that $|p-q'|=\delta_\Omega(p)/2$, which clear satisfies the properties in \rf{eqconc11}.
Further, both $q$ and $q'$ belong to the open connected set
$$U:=\{x\in\Omega:g(p,x)>c_2\,r_0\,\delta_\Omega(p)^{-n}\}$$
for a sufficiently small $c_2>0$.  The fact that $U$ is connected is well known. This follows from the fact that, for any $\lambda>0$, any connected component of $\{g(p,\cdot)>\lambda\}$ should contain $p$. Otherwise there would be a connected component where $g(p,\cdot)-\lambda$ is positive and harmonic with zero boundary values. So, by maximum principle,  $g(p,\cdot)-\lambda$ should equal $\lambda$ in the whole component, which is a contradiction.  So there is only one connected component.

We just let $\gamma$ be a curve contained in $U$. Note that
$$\dist(U,\partial\Omega)\geq c\,r_0^{\frac1\alpha}\,\delta_\Omega(p)^{1-\frac1\alpha}\geq a\,r_0,$$
for a sufficiently small $a>0$
because, by boundary H\"older continuity,
$$g(p,x)\lesssim\left(\frac{\delta_\Omega(x)}{\delta_\Omega(p)}\right)^\alpha\,\frac1{\delta_\Omega(p)^{n-1}}$$
if $\dist(x,\partial\Omega)\leq \delta_\Omega(p)/2$. Further, the fact that 
$g(p,x)\leq c|x-p|^{1-n}$ ensures that $U\subset B(p,C\delta_\Omega(p))$, for a sufficiently big constant $C$ depending on $r/r_0$.

%$$g(p,q')\geq \lambda_0^{-1}\,\frac{\omega^p(B(x_0,r))}{r^{n-1}} \approx_A_1 \frac1{\delta_\Omega(p)^{n-1}}$$
%and, analogously, 
%$$g(p,q)\geq \lambda^{-1}\,\frac{\omega^p(B(x_0,r_0))}{r_0^{n-1}} \approx_{A_1,r/r_0} \frac1{\delta_\Omega(p)^{n-1}},$$
%we infer that both $q$ and $q'$ can be connected by curves $\gamma_q$ and $\gamma_{q'}$ to $p$ contained in $B(x_0,A_1r_0)
%\setminus U_{A_1^{-1}r_0}(\partial\Omega)$. To check this, note that for some constant $c$ depending on $A_1$ and $r/r_0$,
%$$q,q'\in \{x\in\Omega:g(p,x)>c\,\delta_\Omega(p)^{n-1}\},$$
%which is a connected set contained in $B(x_0,A_1r_0)\setminus U_{A_1^{-1}r_0}(\partial\Omega)$, if $A_1$ is big enough. Then we just let $\gamma$ be the union of $\gamma_q$ and $\gamma_{q'}$, and we adjust suitably the parameter $\tau_1$.

\vv
So from now on we assume that $r< 2A_1^{-1} |x_0-p|$. 
By Lemma \ref{lemfacc} we know there exists some point $\wt q\in\Omega$ such that
\begin{equation}\label{eqass12}
\wt q\in B(x_0,\kappa^{-1}r),\quad \delta_\Omega(\wt q)\geq r\geq \kappa\,|x_0-\wt q|\geq \kappa \,\delta_\Omega(\wt q)\geq \kappa\, r
,\quad g(p,\wt q)\geq c\,\frac{\delta_\Omega(\wt q)}{\delta_\Omega(p)^n},
\end{equation}
with $c$ depending on $\kappa$ and $\Lambda$.

Assume that $q$ and $\wt q$ cannot be joined by a curve $\gamma$ as in the statement  of the lemma.  Otherwise, we are done. 
For $t>0$, consider the open set
$$V^t =\bigl\{x\in B(x_0, \tfrac14A_1r):g(p,x)>t \,r_0\,\delta_\Omega(p)^{-n}\bigr\}.$$
We fix $t>0$ small enough such that $q,\wt q\in V^{2t}\subset V^t$. Such $t$ exists by \rf{eqass11} and \rf{eqass12}, and it may depend on $\Lambda,\lambda,r/r_0$.

Let $V_1$ and $V_2$ be the respective components of $V^t$ to which $q$ and $\wt q$ belong.
We have
$$V_1\cap V_2=\varnothing,$$
because otherwise there is a curve contained in $V^t\subset B(x_0,\tfrac14A_1r)$ which connects $q$ and $\wt q$, and further 
this is far away from $\partial\Omega$. Indeed, we claim that
\begin{equation}\label{eqvt5}
\dist(V^t,\partial \Omega)\gtrsim_{A_1,\Lambda,t,r/r_0} r_0.
\end{equation}
To see this, note
that by the H\"older continuity of $g(p,\cdot)$ in $B(x_0,\tfrac12A_1r)$, for all $x\in V^t$, we have
\begin{align*}%\label{eqigu89}
t \,\frac{r_0}{\delta_\Omega(p)^{n}}\leq 
g(p,x) &\lesssim \sup_{y\in B(x_0,\frac12 A_1r)} g(p,y)\,\left(\frac{\delta_\Omega(x)}{A_1r}\right)^\alpha \\
& \leq 
\;\avint_{B(x_0,\frac34 A_1r)} g(p,y)\,dy \left(\frac{\delta_\Omega(x)}{A_1r}\right)^\alpha%\nonumber
\\
& \lesssim_{A_1,\Lambda} \frac{A_1r}
{\delta_\Omega(p)^n}\,\left(\frac{\delta_\Omega(x)}{A_1r}\right)^\alpha,%\nonumber
\end{align*}
 where in the last inequality we used Lemma \ref{lem1'} and that $x_0 \in WA(p,\Lambda)$. This yields our claim.

Next we wish to apply the Alt-Caffarelli-Friedman formula with 
\begin{align*}
u_1(x) &= \chi_{V_1}\,(\delta_\Omega(p)^n\,g(p,x)-t \,r_0)^+,\\
u_2(x) &= \chi_{V_2}\,(\delta_\Omega(p)^n\,g(p,x)-t \,r_0)^+.
\end{align*}
 It is clear that  both satisfy the hypotheses of Theorem \ref{t:ACF}.
For $i=1,2$ and $0<s<A_1r$, we denote
$$J_i(x_0,s) = \frac{1}{s^{2}} \int_{B(x_0,s)} \frac{|\nabla u_i(y)|^{2}}{|y-x_0|^{n-1}}dy,$$
so that
$J(x_0,s) = J_1(x_0,s)\,J_2(x_0,s)$.
 We claim that:
\begin{itemize}
\item[(i)] $J_i(x_0,s)\lesssim_\Lambda 1$ for $i=1,2$ and $0<s<\tfrac14A_1r$.
\item[(ii)] $J_i(x_0,2r) \gtrsim_{\Lambda,\lambda,r/r_0} 1$ for $i=1,2$.
\end{itemize}
The condition (i) follows from \rf{eqjr} and the fact that
\begin{equation}\label{eqgs99}
g(p,y)\lesssim \frac{s}{\delta_\Omega(p)^n}\quad \mbox{ for all $y\in B(x_0,s)$,}
\end{equation}
which holds by Lemma \ref{lem1'} and subharmonicity, since $x_0 \in \WA(p,\Lambda)$.
Concerning (ii), note first that
$$|\nabla u_1(y)| \lesssim \delta_\Omega(p)^n\,\frac{g(p,y)}{\delta_\Omega(y)} \lesssim_{\tau_0} \delta_\Omega(p)^n\,\frac{r_0}{\delta_\Omega(p)^n}=1\quad \mbox{ for all $y\in B(q,\tau_0r_0/2)$},$$
where we first used Cauchy estimates and then the pointwise bounds of $g(\cdot,\cdot)$ 
in \rf{eqgs99} with $s\approx\delta_\Omega(y)$.
 Thus, using also that $q \in V^{2t}$, we infer that $u_1(y)>1.5t\,r_0
$ in some ball $B(q,ctr_0)$ with $c$ possibly depending on $\Lambda,\lambda,r/r_0$. Analogously, we deduce that $u_2(y)>1.5t\,r_0$ in some ball $B(\wt q,ctr_0)$.
Let $B$ be the largest open ball centered at $q$ not intersecting $\partial V_1$ and let $y_{0}\in \partial V_1\cap \partial B$.
 Then, by considering the convex hull $H\subset B$ of $B(q,ctr_0)$ and
$y_0$ and integrating in spherical coordinates (with the origin in $y_0$), one can check that 
$$\int_{H} |\nabla u_1|\,dy \gtrsim_t \,r_0^{n+1}.$$
 An analogous estimate holds for $u_2$, and then it easily follows that
$$J_i(x_0,2r_0)\gtrsim_t 1,$$
which implies (ii). We leave the details for the reader.

From the conditions (i) and (ii) and the fact that $J(x,r)$ is non-decreasing we infer that 
$$J(x_0,s)\approx_{\Lambda,\lambda,r/r_0} 1\quad \mbox{ for $2r<s<\tfrac14A_1r$.}$$
and also 
\begin{equation}\label{eqji*}
J_i(x_0,s)\approx_{\Lambda,\lambda,r/r_0} 1\quad \mbox{ for $i=1,2$ and $2r<s<\tfrac14A_1r$.}
\end{equation}

Assume that $\tfrac14A_1=2^m$ for some big $m>1$. Since $J(x_0,s)$ is non-decreasing we infer that there exists some $h\in [1,m-1]$
such that
$$J(x_0,2^{h+1}r)  \leq C(\Lambda,\lambda,r/r_0)^{1/m} J(x_0,2^hr),$$
because otherwise, by iterating the reverse inequality, we  get a contradiction.
Now from Lemma \ref{lem:ACF2} we deduce that, given any $\ve>0$, for $m$ big enough, 
 there are constant $k_i\approx_{\Lambda,\lambda,r/r_0} 1$ and a unit vector $e$ such that
\begin{equation}\label{eqsum1122}
\|u_1- k_1\,((\cdot-x_0)\cdot e)^+\|_{\infty,B(x_0,2^hr)} + \|u_2- k_2\,((\cdot-x_0)\cdot e)^-\|_{\infty,B(x_0,2^hr)} \leq \ve\,2^h\,r.
\end{equation}
Indeed,  $\|u_i\|_{\infty,B(x_0,2^hr)}\approx_{\Lambda,\lambda,r/r_0} 2^hr$ by 
 \eqref{eqjr} and \rf{eqji*}; $\|u_i\|_{\rm{Lip}^\alpha,B(x_0,2^{h+}r)}\lesssim_{\Lambda,\lambda,r/r_0}(2^{h}r)^{1-\alpha}$ by Lemma \ref{lem333};
and the option (a) in Lemma \ref{lem:ACF2}
cannot hold (since  $\|u_i\|_{\infty,B(x_0,2^hr)}\approx_{\Lambda,\lambda,r/r_0} 2^hr$).

In particular, for $\ve$ small, \rf{eqsum1122} implies that if $q':=x_0 + 2^{h-1}re$, then $u_1(q') \approx_{\Lambda,\lambda,r/r_0} 2^{h-1}r$, and also that 
$$u_1(y) \approx_{\Lambda,\lambda,r/r_0} 2^{h-1}r>0 \quad\mbox{ for all $y\in B(q',2^{h-2}r)$}.$$
Thus  $B(q',2^{h-2}r)\subset\Omega$ and so $q'$ is at a distance at least $2^{h-2}r$ from $\partial\Omega$, and also 
$$g(p,q')\geq\frac{u_1(q')}{\delta_\Omega(p)^n}\approx_{\Lambda,\lambda,r/r_0} \frac{2^h\,r}{\delta_\Omega(p)^n}.$$
Further, since $q$ and $q'$ are both in $V_1$  by definition, there is a curve $\gamma$ which joins $q$ and $q'$ contained in $V_1$ satisfying
$$\dist(\gamma,\partial \Omega)\gtrsim_{A_1,\Lambda,t,r/r_0} r_0,$$
by \rf{eqvt5}.
\end{proof}

\vv

% ***************************************************************************

\section{Types of cubes}\label{sechdld}

From now on we fix $R_0\in\DD_\mu$ and $p\in\Omega$ and we assume that we are under the assumptions
of the Main Lemma \ref{lemhc}.

%\subsection{Cubes of high density and low density}
We need now to define two families
$\HD$ and $\LD$ of high density and low density cubes, respectively.
Let $A\gg1$ be some fixed constant. We 
denote by $\HD$ (high density) the family of maximal cubes $Q\in\DD_\mu$ which are contained in $R_0$
and satisfy
$$\frac{\omega^{p}(2Q)}{\mu(2Q)}\geq A \,\frac{\omega^{p}(2R_0)}{\mu(2R_0)}.$$
We also denote by $\LD$ (low density) the family of maximal cubes $Q\in\DD_\mu$ which are contained in $R_0$
and satisfy
$$\frac{\omega^{p}(Q)}{\mu(Q)}\leq A^{-1} \,\frac{\omega^{p}(R_0)}{\mu(R_0)}$$
(notice that $\omega^{p}(R_0)\approx \omega^{p}(2R_0)\approx 1$ by assumption). Observe that the definition of the family $\HD$ involves the density of $2Q$, while the one of $\LD$ involves the density of $Q$.

We denote
$$B_H=\bigcup_{Q\in \HD} Q \quad\mbox{ and }\quad B_L=\bigcup_{Q\in \LD} Q.$$
\vv

\begin{lemma}\label{lemhd}
We have
$$\mu(B_H) \lesssim \frac1A\,\mu(R_0)\quad \text{ and }\quad \omega^{p}(B_L) \leq \frac1A\,\omega^{p}(R_0).$$
\end{lemma}

\begin{proof}
By Vitali's covering theorem, there exists a subfamily $I\subset\HD$ so that the cubes $2Q$, $Q\in I$, are pairwise disjoint and 
$$\bigcup_{Q\in\HD} 2Q\subset \bigcup_{Q\in I} 6Q.$$
Then, since $\mu$ is doubling, we obtain
\begin{align*}
\mu(B_H) \lesssim \sum_{Q\in I}\mu(2Q) \leq \frac 1A
\sum_{Q\in I}\frac{\omega^{p}(2Q)}{\omega^{p}(2R_0)}\,\mu(2R_0) \lesssim
\frac1A\,\mu(R_0).
\end{align*}

Next we turn our attention to the low density cubes. Since the cubes from $\LD$ are pairwise disjoint, we have
\begin{align*}
\omega^{p}(B_L) = \sum_{Q\in\LD}\omega^{p}(Q) \leq\frac1A
\sum_{Q\in\LD}\frac{\mu(Q)}{\mu(R_0)}\,\omega^{p}(R_0) \leq \frac1A\,\omega^{p}(R_0). 
\end{align*}
\end{proof}

\vv
From the above estimates and the fact that the harmonic measure belongs to weak-$A_\infty$, we infer that
if $A$ is chosen big enough, then 
$$\omega^p(B_H) \leq \ve_0\,\omega^p(2B_{R_0}) \leq \frac14\,\omega^p(R_0)$$
and thus
$$\omega^p(B_H\cup B_L) \leq  \frac14\,\omega^p(R_0)
+ \frac1A\,\omega^p(R_0)
\leq \frac12\,\omega^p(R_0).$$
As a consequence, denoting $G_0= R_0\setminus (B_H\cup B_L))$, we deduce that
$$\omega^p(G_0)\geq\frac12\,\omega^p(R_0)\approx \omega^p(2B_{R_0}),$$
which implies that
$$\mu(G_0)\gtrsim \mu(2B_{R_0})\approx \mu(R_0),$$
again using the fact that $\omega^p$ belongs to weak-$A_\infty$ in $B_{R_0}$.
So we have:

\begin{lemma}\label{lemg0}
Assuming $A$ big enough,
the set $G_0:= R_0\setminus (B_H\cup B_L))$ satisfies
$$\omega^p(G_0)\approx 1\quad \text{ and }\quad \mu(G_0)\approx \mu(R_0),$$
with the implicit constants depending on $C_0$ and the weak-$A_\infty$ condition in $B_{R_0}$.
\end{lemma}

\vv
 We denote by $\fG$ the family of those cubes $Q\in \DD_\mu(R_0)$ which are not contained in 
$\bigcup_{P\in  \HD\cup\LD}P$. In particular, such cubes $Q\in\fG$ do not belong to $\HD\cup\LD$ and thus
\begin{equation}\label{defgr0}
A^{-1}\frac{\omega^p(R_0)}{\mu(R_0)}\leq \frac{\omega^p(Q)}{\mu(Q)}\lesssim\frac{\omega^p(2Q)}{\mu(2Q)}\leq A\,\frac{\omega^p(2R_0)}{\mu(2R_0)}.
\end{equation}
From this fact, it follows easily that $G_0$ is contained in the set $\WA(p,\Lambda)$ defined in Section
\ref{secshortp}, assuming $\Lambda$ big enough, and so Lemma \ref{lemg0} ensures that \rf{eqwap} holds.
\vv

The following lemma is an immediate consequence of Lemma \ref{lemfacc}.

\begin{lemma}\label{lemcork1}
For every cube $Q\in\fG$ there exists some point $x_Q\in 2B_Q\cap\Omega$ such that 
$\delta_\Omega(x_Q)\geq \kappa_0\,\ell(Q)$ and
\begin{equation}\label{eqgg1}
g(p,x_Q)> c_3\,\frac{\ell(Q)}{\mu(R_0)},
\end{equation}
for some $\kappa_0,c_3>0$, which depend on $A$ and on the weak-$A_\infty$ constants in $B_{R_0}$.
\end{lemma}

%{\Bl
%\begin{definition}
If  $x_Q\in 2B_Q\cap\Omega$ and $\delta_\Omega(x_Q)\geq \kappa_0\,\ell(Q)$, we say that $x_Q$ is {\it $\kappa_0$-corkscrew}
for $Q$. If \rf{eqgg1} holds, we say that $x_Q$ is a  {\it $c_3$-good corkscrew} for $Q$.
Abusing notation, quite often we will not write ``for $Q$".

%\end{definition}
%}

\vv
We will need the following auxiliary result:
\begin{lemma}\label{lemclosejumps}
Let $Q\in\DD_\mu$ and let $x_Q$ be a $\lambda$-good $c_4$-corkscrew, for some $\lambda,c_4>0$. Suppose that $\ell(Q)\geq c_5\,\ell(R_0)$.
Then there exists some
$C$-good Harnack chain that joins $x_Q$ and $p$, with $C$ depending on $\lambda,c_5$.
\end{lemma}

\begin{proof}
Consider the open set $U=\{x\in\Omega:g(p,x)>\lambda\,\ell(Q)/\mu(R_0)\}$. This is connected and thus there exists a curve $\gamma
\subset U$ that connects $x_Q$ and $p$.
By H\"older continuity, any point $x\in\Omega$  such that $\delta_\Omega(x)\leq \delta_\Omega(p)/2$,  satisfies
$$g(p,x)\leq c\,\left(\frac{\delta_\Omega(x)}{\ell(R_0)}\right)^\alpha\,\frac1{\ell(R_0)^{n-1}}.$$
Since $g(p,x)>\lambda\,\ell(Q)/\mu(R_0)\gtrsim_{c_5,\lambda}\ell(R_0)^{1-n}$ for all $x\in U$, we deduce that
$\dist(U,\partial\Omega)\geq c_6\,\ell(R_0)$
for some $c_6>0$ depending on $\lambda$ and $c_5$. Thus, 
$$\dist(\gamma,\partial\Omega)\geq c_6\,\ell(R_0).$$

From the fact that $g(p,x)\leq|p-x|^{1-n}$ for all $x\in\Omega$, we infer that any $x\in U$ satisfies
$$\lambda\,\frac{\ell(Q)}{\mu(R_0)} < g(p,x)\leq \frac1{|p-x|^{n-1}}.$$
Therefore, 
$$|p-x|<\left(\frac{\mu(R_0)}{\lambda\,\ell(Q)}\right)^{1/(n-1)}\lesssim_{c_5,\lambda}\ell(R_0).$$
So $U\subset B(p,C_2\,\ell(R_0))$ for some $C_2$ depending on $\lambda$ and $c_5$.
Next we consider a Besicovitch covering of $\gamma$ with balls $B_i$ of radius $c_6\ell(R_0)/2$. By volume considerations, it easily follows that the number of balls $B_i$ is bounded above by some constant $C_3$ depending on $\lambda$ and $c_5$, and thus this
is a $C$-good Harnack chain, with $C=C(\lambda,c_5)$.
\end{proof}

\vv

\begin{lemma}\label{lemshortjumps2}
There exists some constant $\kappa_1$ with $0<\kappa_1\leq\kappa_0$ such that the following holds for all $\lambda>0$. Let $Q\in\fG$, $Q\neq R_0$, and let $x_Q$ be a $\lambda$-good $\kappa_1$-corkscrew.
Then there exists some cube $R\in\fG$ with $Q\subsetneq R\subset R_0$ and $\ell(R)\leq C\,\ell(Q)$
and a $\lambda'$-good $\kappa_1$-corkscrew
$x_R$ such that $x_Q$ and $x_R$ can be joined by a $C'(\lambda)$-good Harnack chain, with $\lambda'>0$ and $C$
depending on $\lambda$.
\end{lemma}

The proof below yields a constant $\lambda'<\lambda$. On the other hand, the lemma ensures that 
$x_R$ is still a $\kappa_1$-corkscrew, which will be important for the arguments to come.

\begin{proof} This follows easily from Lemma \ref{lemshortjumps}. For completeness we will show the details.

 By choosing $\Lambda=\Lambda(A)>0$ big enough, 
$G_0\cap Q\subset\WA(p,\Lambda)$ and thus there exists some $x_0\in Q\cap \WA(p,\Lambda)$. 
We let 
$$\kappa_1=\min(\kappa_0,\kappa),$$ where $\kappa_0$ is defined in Lemma \ref{lemcork1} and
$\kappa$ in Lemma \ref{lemfacc} (and thus it depends only on $A$ and $C_0$).
We apply Lemma \ref{lemshortjumps}
to $x_0$, $q$, with $r_0=3r(B_Q)$, $\lambda_0\approx \lambda$, and $r=4r(B_Q)$. To this end, note that
$$\delta_\Omega(q) \geq \kappa_1\,\ell(Q) =  \kappa_1\,\frac14\,\ell(r(B_Q)) = \kappa_1\,\frac1{12}\,r_0.$$
Hence there exists $q'\in B(x_0,A_1r)$ such that
\begin{equation}\label{eqkap12}
\delta_\Omega(q')\geq \kappa\,|x_0-q'|\geq \kappa\,r,\qquad g(p,q')\geq \lambda_1\,\frac{\delta_\Omega(q')}{\delta_\Omega(p)^n},
\end{equation}
and such that
$q$ and $q'$ can be joined by a curve $\gamma$ such that
\begin{equation}\label{eqgamma83}
\gamma\subset\{y\in B(x_0,A_1r):\dist(y,\partial\Omega)
>a_1\, r_0\},
\end{equation}
with $\lambda_1,A_1,a_1$ depending on on $C_0,A,\lambda,\kappa_1$. Now let $R\in\DD_\mu$ be the cube containing $x_0$
such that
$$\frac12\,r(B_R)< |x_0-q'|\leq r(B_R).$$
Observe that
$$r(B_R) \geq |x_0-q'| \geq r = 4r(B_Q)\quad 
\text{ and }\quad r(B_R)< 2|x_0-q'|\leq 2A_1\,r\lesssim_\lambda \ell(Q).$$
 Also, we may assume that $\ell(R)\leq\ell(R_0)$ because otherwise we have $\ell(Q)\gtrsim A_1\,
\delta_\Omega(p)$ and then the statement in the lemma follows from Lemma \ref{lemclosejumps}.
So we have $Q\subsetneq R\subset R_0$.

From \rf{eqkap12} we get
$$\delta_\Omega(q') \geq\kappa\,|x_0-q'| \geq \frac12\,\kappa\,r(B_R) >\kappa_1\,\ell(R)$$
and 
$$ g(p,q')\geq c\,\lambda_1\,\frac{2\kappa\,\ell(R)}{\mu(R_0)}.
$$

From \rf{eqgamma83} and arguing as in the end of the proof of Lemma \ref{lemclosejumps} we infer that 
$x_Q$ and $x_R$ can be joined by a $C(\lambda)$-good Harnack chain.
\end{proof}

\vv

From now on we will assume that all corkscrew points for cubes $Q\in\fG$ are $\kappa_1$-corkscrews, unless otherwise
stated. 

\vv
% ***************************************************************************

\section{The corona decomposition and the Key Lemma}

\subsection{The corona decomposition}

Recall that the $b\beta$ coefficient of a ball was defined in \rf{defbbeta}.
For each $Q\in\DD_\mu$, we denote
$$b\beta(Q) = b\beta_{\partial\Omega}(100B_Q).$$

Now we fix a constant $0<\ve\ll\min(1,\kappa_1)$.
Given $R\in \DD_\mu(R_0)$, we denote 
by $\sss(R)$ the  maximal family of cubes $Q\in\DD_\mu(R)\setminus \{R\}$ satisfying that either $Q\not \in\fG$ or
$b\beta\bigl(\wh Q\bigr)>\ve$, where $\wh Q$ is the parent of $Q$.
Recall that the family $\fG$ was defined in \rf{defgr0}.
Note that,  by maximality, $\sss(R)$ is a family of pairwise disjoint cubes. 

We define 
$$\tree(R):=\{ Q \in \DD_\mu (R): \nexists \,\, S \in \sss(R) \,\, \textup{such that}\,\, Q\subset S\}.$$ 
 In particular, note that $\sss(R)\not\subset\tree(R)$.% and that if $S \subset Q \in \tree$ 
\vv

We now define the family of the top cubes with respect to $R_0$ as follows:
first we define the families $\ttt_k$ for $k\geq1$ inductively. We set
$$\ttt_1=\{R\in \DD_\mu(R_0)\cap\fG: \ell(R)= 2^{-10}\ell(R_0)\}.$$
Assuming that $\ttt_k$ has been defined, we set
$$\ttt_{k+1} = \bigcup_{R\in\ttt_k}(\sss(R)\cap\fG),$$
and then we define
$$\ttt=\bigcup_{k\geq1}\ttt_k.$$
Notice that the family of cubes $Q\in\DD_\mu(R_0)$ with $\ell(Q)\leq 2^{-10}\ell(R_0)$ which are not contained in any cube $P\in\HD\cup\LD$
is contained in $\bigcup_{R\in\ttt}\tree(R)$,
and this union is disjoint. Also, all the cubes in that union belong to
$\fG$.

The following lemma is an easy consequence of our construction. Its proof is left for the reader.

\begin{lemma}\label{lem5.1}
We have
$$\ttt\subset \fG.$$
Also, for each $R\in\ttt$,
$$\tree(R) \subset \fG.$$
Further, for all $Q\in\tree(R)\cup \sss(R)$,
$$\omega^p(2Q)\leq C\,A\,\frac{\mu(Q)}{\mu(R_0)}.$$
\end{lemma}

Remark that the last inequality holds for any cube $Q\in\sss(R)$ because its parent $\wh Q$ belongs to $\tree(R)$ and so $\wh Q\not \in\HD$, which implies that $\omega^p(2Q)\leq \omega^p(2\wh Q)\lesssim A\,\frac{\mu(\wh Q)}{\mu(R_0)}\approx A\,\frac{\mu(Q)}{\mu(R_0)}.$

Using that $\mu$ is uniformly rectifiable, it is easy to prove that the cubes from $\ttt$ satisfy a Carleson packing condition. This is shown in the next lemma.

\begin{lemma}\label{lempack}
We have
$$\sum_{R\in\ttt}\mu(R) \leq M(\ve)\,\mu(R_0).$$
\end{lemma}

\begin{proof}
For each $Q\in\ttt$ we have
$$\mu(Q) = \sum_{P\in\sss(Q)\cap\fG}\mu(P) + \sum_{P\in\sss(Q)\setminus\fG}\mu(P)
+\mu\biggl(Q\setminus\bigcup_{P\in\sss(Q)} P\biggr).$$
Then we get
\begin{align}\label{eqspl23}
\sum_{Q\in\ttt} \mu(Q) & \leq \sum_{Q\in\ttt}\sum_{P\in\sss(Q)\cap\fG}\mu(P)\\&\quad + \sum_{Q\in\ttt}\sum_{P\in\sss(Q)\setminus\fG}\mu(P)+\sum_{Q\in\ttt}\mu\biggl(Q\setminus\bigcup_{P\in\sss(Q)} P\biggr).\nonumber
\end{align}
Note now that, because of the stopping conditions, for all $Q\in\ttt
$, if $P\in\sss(Q)\cap\fG$, then the parent $\wh P$ of $P$ satisfies
 $b\beta_{\partial\Omega}(100B_{\wh P})>\ve$. Hence, by Theorems 
\ref{teods} and \ref{teo*},
$$\sum_{Q\in\ttt}\sum_{P\in\sss(Q)\cap\fG}\mu(P) \leq \sum_{P\in\DD_\mu(R_0):b\beta_{\partial\Omega}(100B_{\wh P})>\ve}\mu(P)
\leq C(\ve)\,\mu(R_0).$$

On the other hand, the cubes $P\in\sss(Q)\setminus\fG$ with $Q\in\ttt$ do not contain any cube from $\ttt$, by 
construction. Hence, they are disjoint and thus
$$\sum_{Q\in\ttt}\sum_{P\in\sss(Q)\setminus\fG}\mu(P)\leq \mu(R_0).$$
By an analogous reason,
$$\sum_{Q\in\ttt}\mu\biggl(Q\setminus\bigcup_{P\in\sss(Q)} P\biggr)\leq \mu(R_0).$$
By \rf{eqspl23} and the estimates above, the lemma follows.
\end{proof}

\vv

Given a constant $K\gg1$, next we define 
\begin{equation}\label{defg0k}
G_0^K=\biggl\{x\in G_0: \sum_{R\in\ttt}\chi_R(x) \leq K\biggr\},
\end{equation}
By Chebyshev and the preceding lemma, we have
$$\mu(G_0\setminus G_0^K) \leq \mu(R_0\setminus G_0^K) \leq \frac1K \int_{R_0}\sum_{R\in\ttt}\chi_R\,
d\mu \leq \frac{M(\ve)}K\,\mu(R_0).$$
Therefore, if $K$ is chosen big enough (depending on $M(\ve)$ and the constants on the weak-$A_\infty$ condition), by Lemma \ref{lemg0} we get
$$\mu(G_0\setminus G_0^K)\leq \frac12\,\mu(G_0),$$
and thus
$$\mu(G_0^K)\geq\frac12\,\mu(G_0)\gtrsim \mu(R_0).$$

\vv
We distinguish now two types of cubes from $\ttt$. We denote by $\ttt_a$ the family of cubes $R\in\ttt$ such that
$\tree(R)=\{R\}$, and we set $\ttt_b=\ttt\setminus \ttt_a$. Notice that, by construction, if $R\in\ttt_b$, then
$b\beta(R)\leq\ve$. On the other hand, this estimate may fail if $R\in\ttt_a$.

\vv

% ***************************************************************************

\subsection{The truncated corona decomposition}\label{secnnn}

For technical reasons, we need now to define a truncated version of the previous corona decomposition.
We fix a big natural number $N\gg 1$. Then we let $\ttt^{(N)}$ be the family of the cubes
from $\ttt$ with side length larger than $2^{-N}\ell(R_0)$. Given $R\in\ttt^{(N)}$ we let $\tree^{(N)}_b(R)$ be the 
subfamily of the cubes from $\tree(R)$  with side length larger than $2^{-N}\ell(R_0)$, and we let $\sss^{(N)}(R)$
be a maximal subfamily from $\sss(R)\cup\DD_{\mu,N}(R_0)$, where $\DD_{\mu,N}(R_0)$ is the subfamily of
the cubes from $\DD_\mu(R_0)$ with side length $2^{-N}\ell(R_0)$.  We also denote $\ttt_a^{(N)} =\ttt^{(N)}\cap \ttt_a$ and 
$\ttt_b^{(N)} =\ttt^{(N)}\cap \ttt_b$.

Observe that, since $\ttt^{(N)}\subset \ttt$, we also have
$$\sum_{R\in\ttt^{(N)}}\chi_R(x) \leq \sum_{R\in\ttt}\chi_R(x) \leq K\quad
\mbox{ for all $x\in G_0^K$.}$$

\vv

% ****************************************************************************

\subsection{The Key Lemma}

The main ingredient for the proof of the Main Lemma \ref{lemhc}  is the following result.

\begin{lemma}[Key Lemma] \label{keylemma3}
Given $\eta\in (0,1)$ and
$\lambda\in (0,c_3]$ (with $c_3$ as in \rf{eqgg1}),
there exists an exceptional family $\mathsf{Ex}(R)\subset \sss(R)\cap\fG$ 
satisfying
$$\sum_{P\in\mathsf{Ex}(R)}\mu(P)\leq \eta\,\mu(R)$$
such that,
for every $Q\in \sss(R)\cap\fG\setminus \mathsf{Ex}(R)$, 
 any $\lambda$-good corkscrew for $Q$ can be joined to some $\lambda'$-good corkscrew for $R$ by a $C(\lambda,\eta)$-good Harnack chain, with $\lambda'$ depending on $\lambda,\eta$.
\end{lemma}

This lemma will be proved in the next Sections \ref{sec6} and \ref{sec7}. Using this result, in Section
\ref{sec8} we will build the required carrot curves for the  Main Lemma \ref{lemhc}, which join the pole $p$ to points from a suitable big piece
of $R_0$. If the reader prefers to see how this is applied before its long proof, they may go directly to Section \ref{sec8}. A key point in the Key Lemma is that the constant $\ve$ in the definition of the stopping cubes of the corona decomposition does not depend on the constants $\lambda$ or $\eta$ above.

To prove the Key Lemma \ref{keylemma3} we will need first to introduce the notion of ``cubes with well
separated big corkscrews" and we will split $\tree^{(N)}(R)$ into subtrees by introducing an additional stopping condition involving this type of cubes. Later on, in Section \ref{sec6} we will prove the ``Geometric Lemma", which
relies on a geometric construction which plays a fundamental role in the proof of the Key Lemma.

% ***************************************************************************

\vv

\subsection{The cubes with well separated big corkscrews}\label{subsepcork}

Let $Q\in\DD_\mu$ be a cube such that $b\beta(Q)\leq C_4\ve$. For example, $Q$ might be a cube from $Q\in\tree^{(N)}(R)\cup\sss^{(N)}(R)$, with  $R\in\ttt^{(N)}_b$ (which in particular implies that $b\beta(R)\leq \ve$). We denote by $L_Q$ a best approximating $n$-plane for $b\beta(Q)$,
and we choose $x_Q^1$ and $x_Q^2$ to be two fixed points in $B_Q$ such that $\dist(x_Q^i,L_Q) = r(B_Q)/2$ and lie
in different components of $\R^{n+1}\setminus L_Q$. So $x_Q^1$ and $x_Q^2$ are corkscrews for $Q$. We will call them ``big corkscrews".

Since any corkscrew $x$ for $Q$ satisfies $\delta_\Omega(x)\geq \kappa_1\,\ell(Q)$ and we have chosen
$\ve\ll\kappa_1$, it turns out that
$$\dist(x,L_Q)\geq \frac12\,\kappa_1\,\ell(Q)\gg \ve\,\ell(Q).$$
As a consequence, $x$ can be joined either to $x_Q^1$ or to $x_Q^2$ by a $C$-good Harnack chain, with 
$C$ depending only on $n,C_0,\kappa_1$, and thus only on $n$, $C_0$ and the weak-$A_\infty$ constants in $B_{R_0}$.
The following lemma follows by the same reasoning:

\begin{lemma}\label{lembigcork}
Let $Q,Q'\in\DD_\mu$ be cubes such that $b\beta(Q),b\beta(Q')\leq C_4\ve$ and $Q'$ is the parent of $Q$. Let $x_Q^i,x_{Q'}^i$, for $i=1,2$, be big corkscrews for $Q$ and $Q'$ respectively. Then, after relabelling the corkscrews if necessary, $x_Q^i$ can be joined to
$x_{Q'}^i$ by a $C$-good Harnack chain, with
$C$ depending only on $n,C_0,\kappa_1$.
\end{lemma}

Given $\Gamma>0$, we will write $Q\in \WSBC(\Gamma)$ (or just $Q\in\WSBC$, which stands for ``well separated big corkscrews") if $b\beta(Q)\leq C_4\ve$ and the big corkscrews $x_Q^1$, $x_Q^2$
 can {\em not} be joined by any $\Gamma$-good Harnack chain. The parameter $\Gamma$ will be chosen below. For the moment, let us say
 that $\Gamma^{-1}\ll\ve$. The reader should think that in spite of $b\beta(Q)\leq C_4\ve$, the possible existence of ``holes of size $C\,\ve\ell(Q)$ in $\supp\mu$''
 makes possible the connection of the big corkscrews by means of $\Gamma$-Harnack chains passing through these holes.
Note that if $Q\not\in \WSBC(\Gamma)$, then any pair of corkscrews for $Q$ can be connected by a $C(\Gamma)$-good Harnack chain, since any of these corkscrews can be joined by a good chain to one of the big corkscrews for $Q$, as mentioned above.

% ***************************************************************************
\vv

\subsection{The tree of cubes of type $\WSBC$ and the subtrees}\label{sec5.5}

Given $R\in\ttt^{(N)}_b$, denote
by $\sss_\WSBC(R)$ the maximal subfamily of cubes from $Q\in\DD_\mu(R)$ which satisfy that either
\begin{itemize}
\item $Q\not\in\WSBC(\Gamma)$, or
\item $Q\not\in\tree^{(N)}(R)$.
\end{itemize}
Also, denote by $\tree_\WSBC(R)$ the cubes from $\DD_\mu(R)$ which are not strictly contained in any cube from 
$\sss_\WSBC(R)$. So this tree is empty if $R\not\in\WSBC(\Gamma)$.

Observe that if $Q\in\sss_\WSBC(R)$, it may happen that $Q\not\in\WSBC(\Gamma)$. However, 
unless $Q=R$, it holds that $Q\in\WSBC(\Gamma')$, with $\Gamma'>\Gamma$ depending only on $\Gamma$ and $C_0$
(because the parent of $Q$ belongs to $\WSBC(\Gamma)$).

For each $Q\in\sss_\WSBC(R)\setminus\sss(R)$, we denote 
$$\stree(Q) = \DD_\mu(Q)\cap\tree^{(N)}(R),\qquad \ssss(Q)=\sss(R)\cap\DD_\mu(Q).$$
So we have
$$\tree^{(N)}(R) = \tree_\WSBC(R) \cup \bigcup_{Q\in\sss_\WSBC(R)}\stree(Q),$$
and the union is disjoint.  Observe also that we have the partition
\begin{equation}\label{eqstop221}
\sss(R) = \bigl(\sss_\WSBC(R)\cap \sss(R)\bigr) \cup \bigcup_{Q\in\sss_\WSBC(R)\setminus\sss(R)}\ssss(Q).
\end{equation}

% ***************************************************************************

\vv

\section{The geometric lemma}\label{sec6}

\subsection{The geometric lemma for the tree of  cubes of type $\WSBC$}

Let $R\in\ttt^{(N)}_b$ and suppose that $\tree_\WSBC(R)\neq \varnothing$.
We need now to define a family $\eend(R)$ of cubes from $\DD_\mu$, which in a sense can be considered as a 
regularized  version of $\sss(R)$. The first step consists of introducing the following auxiliary function:
$$d_R(x) :=\inf_{Q\in \tree_\WSBC(R)}(\ell(Q) + \dist(x,Q)),\quad\mbox{ for $x\in\R^{n+1}$.}$$
Observe that $d_R$ is $1$-Lipschitz.

For each $x\in \partial\Omega$ we take the largest cube $Q_x\in\DD_\mu$ 
such that $x\in Q_x$ and
\begin{equation}\label{eqdefqx}
\ell(Q_x) \leq \frac1{300}\,\inf_{y\in Q_x} d_R(y).
\end{equation}
We consider the collection of the different cubes $Q_x$, $x\in \partial\Omega$, and we denote it by $\eend(R)$.

\vv

\begin{lemma}\label{lem74}
Given $R\in\ttt^{(N)}_b$, the cubes from $\eend(R)$ are pairwise disjoint and satisfy the following properties:
\begin{itemize}
\item[(a)] If $P\in\eend(R)$ and $x\in 50B_P$, then $100\,\ell(P)\leq d_R(x) \leq 900\,\ell(P)$. 

\item[(b)] There exists some absolute constant $C$ such that if $P,P'\in\eend(R)$ and $50B_P\cap 50B_{P'}
\neq\varnothing$, then
$C^{-1}\ell(P)\leq \ell(P')\leq C\,\ell(P).$
\item[(c)] For each $P\in \eend(R)$, there at most $N$ cubes $P'\in\eend(R)$ such that
$50B_P\cap 50B_{P'}
\neq\varnothing,$
 where $N$ is some absolute constant.
 
\item[(d)] If $P\in\eend(R)$ and $\dist(P,R)\leq 20\,\ell(R)$, then there exists some $Q\in\tree_\WSBC(R)$
such that $P\subset 22Q$ and $\ell(Q)\leq 2000\,\ell(P)$.
 %\item[(d)] If $x\not\in B(x_0,\frac1{8} K r_0)$, then $d(x)\approx |x-x_0|$. Thus,
 %if $P\in\eend(R)$ and $B(z_{P},50\ell(P))\not\subset  B(x_0,\frac1{8} K r_0)$, then $\ell(P)\gtrsim  K r_0$.
\end{itemize}
\end{lemma}

\begin{proof}
The proof is a routine task. For the reader's convenience we show the details..
To show (a), consider  $x\in 50 B_P$. Since $d_R(\cdot)$ is $1$-Lipschitz
and, by definition, $d_R(z_{P})\geq 300\,\ell(P)$, we have
$$d_R(x)\geq d_R(z_{P}) - |x-z_{P}| \geq d_R(z_{P}) - 50\,r(B_P)\geq 300 \,\ell(P)-200\,\ell(P)=100\,\ell(P) .$$

To prove the converse inequality, by the definition of $\eend(R)$, there exists some $z'\in
\wh P$, the parent of $P$, such that 
$$d_R(z')\leq 300\,\ell(\wh P)\ = 600\,\ell(P).$$
Also, we have
$$|x-z'|\leq |x-z_{P}| + |z_{P}-z'|\leq 50\,r(B_P) + 2\ell(P)\leq 300\,\ell(P).$$
Thus,
$$d_R(x)\leq d_R(z') + |x-z'| \leq (600+ 300)\,\ell(P).$$

The statement (b) is an immediate consequence of (a), and (c) follows easily from (b).
To show (d), observe that, for any $S\in\tree_\WSBC(R)$,
$$\ell(P)\leq\frac{d(z_P)}{300}\leq \frac{\ell(S) + \dist(z_P,S)}{300}
\leq \frac{\ell(P) + \ell(S) + \dist(P,S)}{300}.$$
Thus,
$$\ell(P)\leq\frac{\ell(S) + \dist(P,S)}{299}.$$
In particular, choosing $S=R$, we deduce
$$\ell(P)\leq\frac{\ell(R) + \dist(P,R)}{299}
\leq \frac{21}{299}\,\ell(R)\leq \ell(R),$$
and thus, using again that $\dist(P,R)\leq20\ell(R)$, it follows that 
$P\subset 22R$. Let $S_0\in\tree_\WSBC(R)$ be such that $d(z_P)=\ell(S_0) + \dist(z_P,S_0)$, and let $Q\in\DD_\mu$ be be the smallest cube such that $S_0\subset Q$ and $P\subset 22Q$. Since $S_0\subset R$ and $P\subset 22R$, 
we deduce that $S_0\subset Q\subset R$, implying that $Q\in\tree_\WSBC(R)$. 

So it just remains to check that
$\ell(Q)\leq 2000\,\ell(P)$. To this end, consider a cube $\wt Q\supset S_0$ such that
$$\ell(P)+\ell(S_0)+\dist(P,S_0)\leq \ell(\wt Q)\leq 2\bigl(\ell(P)+\ell(S_0)+\dist(P,S_0)\bigr).$$
From the first inequality, it is clear that $P\subset 2\wt Q$ and then, by the definition of $Q$, we infer that $Q\subset\wt Q$. This inclusion and the second inequality above imply that 
$$\ell(Q)\leq\ell(\wt Q) \leq 2\bigl(2\ell(P)+\ell(S_0)+\dist(z_P,S_0)\bigr) = 4\ell(P) + 2\,d_R(z_P).$$
By (a) we know that $d_R(z_P)\leq900\,\ell(P)$, and so we derive $\ell(Q) \leq 2000\,\ell(P)$.
\end{proof}

\vv

\begin{lemma}\label{lemc6}
Given $R\in\ttt^{(N)}_b$, if $Q\in\eend(R)$ and $\dist(P,R)\leq 20\,\ell(R)$, then $b\beta(Q)\leq C\,\ve$ and $Q\in \WSBC(\Gamma')$, with $\Gamma'=c_6\,\Gamma$, for some absolute
constants $C,c_6>0$.
\end{lemma}

\begin{proof}
This immediate from the fact that, by (d) in the previous lemma, there exists some cube $Q'\in\tree_\WSBC(R)$
such that $Q\subset 22Q'$ and $\ell(Q')\leq 2000\,\ell(Q)$, so that $b\beta(Q')\leq\ve$ and 
$Q'\in \WSBC(\Gamma)$.
\end{proof}

\vv

Next we
consider the following Whitney decomposition of $\Omega$: we let $\WW$ be a family of dyadic cubes from $\R^{n+1}$, contained in $\Omega$, with disjoint interiors, such that
$$\bigcup_{I\in\WW} I = \Omega,$$
and such that moreover there are
 some constants $M_0>20$ and $D_0\geq1$ satisfying the following for every $I\in\WW$:
\begin{itemize}
\item[(i)] $10I \subset \Omega$;
\item[(ii)] $M_0 I \cap \partial\Omega \neq \varnothing$;
\item[(iii)] there are at most $D_0$ cubes $I'\in\WW$
such that $10I \cap 10I' \neq \varnothing$. Further, for such cubes $I'$, we have $\ell(I')\approx
\ell(I)$, where $\ell(I')$ stands for the side length of $I'$.
\end{itemize}
From the properties (i) and (ii) it is clear that $\dist(I,\partial\Omega)\approx\ell(I)$. We assume that
the Whitney cubes are small enough so that
\begin{equation}\label{eqeq29}
\diam(I)< \frac1{100}\,\dist(I,\partial\Omega).
\end{equation}
This can be achieved by replacing each cube $I\in\WW$ by its descendants $I'\in\DD_k(I)$, for some fixed $k\geq1$, if necessary. 

%{\Rd Here we use for the balls $B_I$ the same notation as we did in the construction of the dyadic cubes for the balls that contain the dyadic cubes. I hope this will not be confusing later}
For each $I\in\WW$, we denote by $B^I$ a ball concentric with $I$ and radius $C_5\ell(I)$, where $C_5$ is a
universal constant big enough
so that
$$g(p,x) \lesssim \frac{\omega^p(B^I)}{\ell(I)^{n-1}}\quad\mbox{ for all $x\in 4I$.}$$
Obviously, the ball $B^I$ intersects $\partial\Omega$, and the family $\{B^I\}_{I\in\WW}$ does not have
finite overlapping.

\vv
To state the Geometric Lemma we need some additional notation.
Given a cube $R'\in\tree_\WSBC(R)$,
we denote by $\wt\tree_\WSBC(R')$ the family of cubes from $\DD_\mu$ with side length at most $\ell(R')$ which 
are contained in $100B_{R'}$ and are
not contained in any cube from $\eend(R)$. We also denote by $\wt\eend(R')$ the subfamily of the cubes from $\eend(R)$ which 
are contained in some cube from $\wt\tree_\WSBC(R')$. Note that $\wt\tree_\WSBC(R')$ is not a tree, in general, but a union of trees.

\vv

\begin{lemma}[Geometric Lemma]\label{lemgeom}
Let $0<\gamma<1$, and assume that the constant $\Gamma=\Gamma(\gamma)$ in the definition of $\WSBC$ is big enough.
Let $R\in\ttt^{(N)}_b\cap\WSBC(\Gamma)$ and let $R'\in\tree_\WSBC(R)$ be such that
$\ell(R')=2^{-k_0}\ell(R)$, with $k_0=k_0(\gamma)\geq1$ big enough. 
Then there are two connected open sets $V_1,V_2\subset CB_{R'}\cap\Omega$ with disjoint closures which satisfy the following properties: 
\begin{itemize}
\item[(a)] There are subfamilies $\WW_i\subset\WW$ such that
$V_i=\bigcup_{I\in \WW_i} 1.1\mathring I.$ %Further, if we set $\wt V_i=\bigcup_{I\in \WW_i} 3\mathring I,$ then $\wt V_1\cap\wt V_2=\varnothing$.

\item[(b)] Each $V_i$ contains a ball $B_i$ with $r(B_i)\approx\ell(R')$, and each corkscrew point for $R'$ contained in $2B_{R'}\cap V_i$ can be joined
to the center $z_i$ of $B_i$ by a good Harnack chain contained in $V_i$. Further, any point $x\in V_i$
can be joined to $z_i$ by a good Harnack chain (not necessarily contained in $V_i$).

\item[(c)] For each $Q\in\tree_\WSBC(R)\cap\DD_\mu(R')$ there are big corkscrews $x_Q^1\in V_1\cap 2B_Q$ and $x_Q^2\in V_2\cap 2B_Q$, and if $\wh Q$ is an
ancestor of $Q$ which also belongs to $\tree_\WSBC(R)\cap\DD_\mu(R')$, then $x_Q^i$ can be joined to $x_{\wh Q}^i$ by a good Harnack
chain, for each $i=1,2$.

\item[(d)] 
$(\partial V_1\cup\partial V_2)\cap 10B_{R'}
\subset \bigcup_{P\in\wt\eend(R')}2B_P$.

\item[(e)] If $P\in\wt\eend(R')$ is such that $2B_P\cap10B_{R'}\neq\varnothing$, then $\partial V_i\cap 2B_P$ is contained
in the union of cubes of a subfamily  $\WW_P\subset\WW$ such that 
\begin{itemize}
\item[(i)]  $$m_{4I} g(p,\cdot)\leq \gamma\,\frac{\ell(P)}{\mu(R_0)}\quad \mbox{ for each $I\in\WW_P$,}$$ and
\item[(ii)] $$\sum_{I\in\WW_P} \ell(I)^n \lesssim \ell(P)^n
\quad\text{ and }\quad \sum_{I\in\WW_P} \omega^p(B^I) \lesssim \omega^p(CB_P),$$
for some universal constant $C>1$.
\end{itemize}
\end{itemize} 
The constants involved in the Harnack chain and corkscrew conditions may depend on $\ve$, $\Gamma$, and $\gamma$.\footnote{To guarantee the existence of the sets $V_i$ and the fact that they are
contained in $\Omega$ we use the assumption that $\Omega=(\partial\Omega)^c$.}
\end{lemma}

%The fact that the union in (c) runs over  the cubes in $\eend(R)$, which are much
%smaller than the ones from $\sss_\WSBC(R)$, will play an essential role later, in the proof
%of the Key Lemma \ref{keylemma}.

\vv
% ***************************************************************************

\subsection{Proof of the Geometric Lemma \ref{lemgeom}}

In this whole subsection we fix $R\in\ttt^{(N)}_b$ and we assume $\tree_\WSBC(R)\neq \varnothing$, as in 
Lemma \ref{lemgeom}. We let $R'\in\tree_\WSBC(R)$ be such that
$\ell(R')=2^{-k_0}\ell(R)$, with $k_0=k_0(\gamma)\geq1$ big enough, as in Lemma \ref{lemgeom},
and we consider the associated families 
 $\wt\tree_\WSBC(R')$ and  $\wt\eend(R')$.

\begin{remark}\label{rem**}
By arguments analogous to the ones in Lemma \ref{lemc6}, it follows easily that if $Q\in\wt\tree_\WSBC(R')$, for
$R'\in\tree_\WSBC(R)$ such that $\ell(R')=2^{-k_0}\ell(R)$,
 then there exists some cube $S\in \tree_\WSBC(R)$ such that $Q\subset 22S$ and $\ell(S)\leq2000\ell(Q)$. This implies that 
 $b\beta(Q)\leq C\,\ve$ and $Q\in \WSBC(c_6\Gamma)$ too.
\end{remark}

In order to define the open sets $V_1$, $V_2$ described in the lemma, first we need to associate
some open sets $U_1(Q)$, $U_2(Q)$ to each $Q\in\wt\tree_\WSBC(R')\cup\wt\eend(R')$. We distinguish two cases:
\begin{itemize}
\item For $Q\in\wt\tree_\WSBC(R')$, we let $\mathcal J_i(Q)$ be the family of Whitney cubes $I\in\WW$ which intersect
$$\{y\in\ 20B_Q:\dist(y,L_Q)>\ve^{1/4}\,\ell(Q)\}$$
and are contained in the same connected component of $\R^{n+1}\setminus L_Q$ as $x_Q^i$, and then we set
$$U_i(Q) = \bigcup_{I\in \mathcal J_i(Q)} 1.1\mathring I.$$

\item For $Q\in\wt\eend(R')$ the definition of $U_i(Q)$ is more elaborated. First we consider an auxiliary ball
$\wt B_Q$, concentric with $B_Q$, such that $19B_Q\subset \wt B_Q\subset 20B_Q$ and having thin boundaries for $\omega^p$. This means that, for some absolute constant $C$,
\begin{equation}\label{eqthinbd}
\omega^p\bigl(\bigl\{x\in 2\wt B_Q:\dist(x,\,\partial \wt B_Q)\leq t\,r(\wt B_Q)\bigr\}\bigr) \leq C\,{\Rd t}\,\omega^p(2\wt B_Q)
\quad\mbox{ for all $t>0$.}
\end{equation}
The existence of such ball $\wt B_Q$ follows by well known arguments (see for example \cite[p.370]{Tolsa-llibre}).

Next we denote by $\mathcal J(Q)$ the family of Whitney cubes $I\in\WW$ which
 intersect $\wt B_Q$ and satisfy $\ell(I)\geq \theta\,\ell(Q)$ for $\theta\in(0,1)$ depending on $\gamma$
(the reader should think that $\theta\ll\ve$ and that $\theta=2^{-j_1}$ for some $j_1\gg1$), and we set
\begin{equation}\label{eqdefuq}
U(Q) = \bigcup_{I\in \mathcal J(Q)} 1.1\mathring I.
\end{equation}
For a fixed $i=1$ or $2$,
let $\{D_j^i(Q)\}_{j\geq0}$ be the connected components of $U(Q)$ which satisfy one of the following properties:
\begin{itemize}
\item either $x_Q^i\in D_j^i(Q)$ (recall that $x_Q^i$ is a big corkscrew for $Q$), or
\item there exists some $y\in D_j^i(Q)$ such that $g(p,y)> \gamma\,\ell(Q)\,\mu(R_0)^{-1}$ and there is a $C_6(\gamma,\theta)$-good
Harnack chain that joins $y$ to $x_Q^i$, for some constant $C_6(\gamma,\theta)$ to be chosen below.
\end{itemize}
Then we let $U_i(Q)=\bigcup_j D_j^i(Q)$. After reordering the sequence, we assume that $x_Q^i\in D_0^i(Q)$.
% We let $\mathcal J_1(Q)$ be the subfamily of cubes from $\mathcal J(Q)$ contained in $U_i(Q)$.
\end{itemize}

In the case $Q\in\wt\tree_\WSBC(R')$, from the definitions, it is clear that the sets $U_i(Q)$ are open and connected and 
\begin{equation}\label{eqinter1}
\overline{U_1(Q)}\cap \overline{U_2(Q)}=\varnothing.
\end{equation}
 In the case $Q\in\wt\eend(R')$, the sets $U_i(Q)$ may fail to be connected. However, \rf{eqinter1} still holds if $\Gamma$ is chosen big enough (which will be the case). Indeed, if some component $D_j^i$ can be joined by $C_6(\gamma,\theta)$-good
Harnack chains both to $x_Q^1$ and $x_Q^2$, then there is a $C(\gamma,\theta)$-good Harnack chain that joins $x_Q^1$ to 
$x_Q^2$, and thus  $Q$ does not belong to $\WSBC(c_6\Gamma)$ if $\Gamma$ is taken big enough, which cannot happen by Lemma \ref{lemc6}.
Note also that the two components of
$$\{y\in\wt B_Q:\dist(y,L_Q)>\ve^{1/2}\,\ell(Q)\}$$
are contained in $D_0^1(Q)\cup D_0^2(Q)$, because  $b\beta(Q)\leq C\ve$ and we assume $\theta\ll\ve$.

\vv

The following is immediate:

\begin{lemma}\label{lemgeom3}
Assume that we relabel appropriately the sets $U_i(P)$ and corkscrews $x_P^i$ for $P\in\wt\tree_\WSBC(R')\cup\wt\eend(R')$. Then for all $Q,\wh Q\in\wt\tree_\WSBC(R')\cup\wt\eend(R')$ such that $\wh Q$ is the parent of $Q$ we have
\begin{equation}\label{eqlabel}
\bigl[x_Q^1,x_{\wh Q}^1\bigr]\subset U_1(Q) \cap 
U_1(\wh Q)\quad \text{ and }\quad \bigl[x_Q^2,x_{\wh Q}^2\bigr]\subset U_2(Q) \cap 
U_2(\wh Q).
\end{equation}
Further, 
$$\dist\bigl([x_Q^i,x_{\wh Q}^i],\partial\Omega\bigr)\geq c\,\ell(Q)\quad\mbox{ for $i=1,2$,}$$
where $c$ depends at most on $n$ on $C_0$.
\end{lemma}

The labelling above can be chosen inductively. First we fix the sets $U_i(T)$ and corkscrews $x_{T}^i$
for every maximal cube $T$ from $\wt\tree_\WSBC(R')$ (contained in $100B_{R'}$ and with side length equal to $\ell(R')$). Further we assume
that, for any maximal cube $T$, the corkscrew $x^i_T$ is at the same side of $L_{R'}$ as $x_{R'}^i$, for each $i=1,2$ (this property will be used below).
Later we label the sons of each $T$ so that \rf{eqlabel} holds for any son $Q$ of $T$. Then we proceed with the grandsons of
$T$, and so on. We leave the details for the reader. 

The following result will be used later to prove the property (e)(i).

\begin{lemma}\label{lemei}
Suppose that the constant $k_0(\gamma)$ in  Lemma \ref{lemgeom} is big enough.
Let $Q\in\wt\eend(R')$ and assume $\theta$ small enough and $C_6(\gamma,\theta)$ big enough in the definition of $U_i(Q)$. If $y\in \wt B_Q$ satisfies
$g(p,y)> \gamma\,\ell(Q)\,\mu(R_0)^{-1}$, then $y\in U_1(Q)\cup U_2(Q)$.
\end{lemma}

\begin{proof}
By the definition of $U_i(Q)$, it suffices to show that $y$ belongs to some component $D_j^i(Q)$  and that there is a $C_6(\gamma,\theta)$-good
Harnack chain that joins $y$ to $x_Q^i$.
To this end, observe that by the boundary H\"older continuity of $g(p,\cdot)$,
$$\gamma\,\frac{\ell(Q)}{\mu(R_0)} \leq g(p,y)\leq C\,\left(\frac{\delta_\Omega(y)}{\ell(Q)}\right)^\alpha\,m_{30B_Q}g(p,\cdot)
\leq C\,\left(\frac{\delta_\Omega(y)}{\ell(Q)}\right)^\alpha\,\frac{\ell(Q)}{\mu(R_0)},$$
where in the last inequality we used Lemma \ref{lem1'}. Thus,
$$\delta_\Omega(y)\geq c\,\gamma^{1/\alpha}\,\ell(Q),$$
and if $\theta$ is small enough, then $y$ belongs to some connected component of the set $U(Q)$ in \rf{eqdefuq}. By Lemma \ref{lem74}(d) there is a cube $Q'\in\tree_\WSBC(R)$
such that $Q\subset 22Q'$ and $\ell(Q')\approx\ell(Q)$. In particular, $\WA(p,\Lambda)\cap Q'\supset G_0\cap Q'\neq\varnothing$
and thus, by applying Lemma
\ref{lemshortjumps} with $q=y$ and $r_0=Cr(B_Q)$ (for a suitable $C>1$), 
%\mih{Why $\WA(p,\Lambda)\cap\wh Q=\varnothing$?}
it follows that there exists a $\kappa_1$-corkscrew $y'\in C(\gamma)\,B_Q$, with $C(\gamma)>20$ say, such that
$y$ can be joined to $y'$ by a $C'(\gamma)$-good Harnack chain. Assuming that the constant $k_0(\gamma)$ in Lemma \ref{lemgeom} is big enough,
it turns out that $y'\in CB_{Q''}$ for some $Q''\in \tree_\WSBC(R)$ such that $22Q''\supset Q$. Since all the cubes $S$ such that $Q\subset S\subset 22Q''$ satisfy
$b\beta(S)\leq C\,\ve$, by applying Lemma \ref{lembigcork} repeatedly, it follows that $y'$ can be joined either to $x_Q^1$
or $x_Q^2$ by a $C''(\gamma)$-good Harnack chain. Then, joining both Harnack chains, it follows that $y$ can be joined either
to $x_Q^1$ or $x_Q^2$ by a $C'''(\gamma)$-good Harnack chain. So $y$ belongs to one of the components $D_j^i$, assuming
$C_6(\gamma,\theta)$ big enough.
\end{proof}

\vv

From now on we assume $\theta$ small enough and $C_6(\gamma,\theta)$ big enough so that the preceding
lemma holds. Also, we assume $\theta\ll \ve^4$.
We define
$$V_1 = \bigcup_{Q\in\wt\tree_\WSBC(R')\cup\wt\eend(R')} U_1(Q),\qquad V_2 = \bigcup_{Q\in\wt\tree_\WSBC(R')\cup\wt\eend(R')} U_2(Q).$$
Next we will show that 
$$\overline {V_1}\cap \overline{V_2}=\varnothing.$$
Since the number of cubes $Q\in\wt\tree_\WSBC(R')\cup\wt\eend(R')$ is finite (because of the truncation in the
corona decomposition), this is a consequence of the following:

\begin{lemma}
Suppose $\Gamma$ is big enough in the definition of $\WSBC$ (depending on $\theta$). For all $P,Q\in\wt\tree_\WSBC(R')\cup\wt\eend(R')$, we have
$$\overline{U_1(P)} \cap \overline{U_2(Q)}=\varnothing.$$
\end{lemma}

\begin{proof}
We suppose that $\ell(Q)\geq \ell(P)$
We also assume that $\overline{U_1(P)} \cap \overline{U_2(Q)}\neq\varnothing$ and then we will get a contradiction. Notice first that if $\ell(P)=\ell(Q)=2^{-j}\ell(R')$ for some $j\geq0$, then the corkscrews $x_P^i$ and $x_Q^i$ are at the same side of $L_Q$ for each $i=1,2$. This follows easily by induction on $j$.

\vv
\noi {\bf 1.}
Suppose first that $P,Q\in\wt\tree_\WSBC(R')$. 
Since the cubes from 
$\mathcal J_2(Q)$ have side length at least $c\,\ve^{1/4}\,\ell(Q)$, it follows that at least one of the cubes 
from $\mathcal J_1(P)$ has side length at least $c'\,\ve^{1/4}\,\ell(Q)$, which implies that
$\ell(P)\geq c''\,\ve^{1/4}\,\ell(Q)$, by the construction of $U_1(P)$.

Since $U_1(P) \cap U_2(Q)\neq\varnothing$, there exists some curve $\gamma=\gamma(x_P^1,x_Q^2)$ that joins $x_P^1$ and $x_Q^2$
such that $\dist(\gamma,\partial\Omega)\geq c\,\ve^{1/2}\,\ell(Q)$ because all the cubes from 
$\mathcal J_2(Q)$ have side length at least $c\,\ve^{1/4}\,\ell(Q)$, and the ones from $\mathcal J_1(P)$ 
have side length $\geq c\,\ve^{1/4}\,\ell(P)\geq c\,\ve^{1/2}\,\ell(Q)$. 

Let $\wh P$ be the ancestor of $P$ such that $\ell(\wh P)=\ell(Q)$. From the fact that 
$\overline{U_1(P)} \cap \overline{U_2(Q)}\neq\varnothing$, we deduce that $20B_P\cap 20B_Q\neq\varnothing$ and thus 
$20B_{\wh P}\cap 20B_Q\neq\varnothing$, and so $20B_{\wh P}\subset 60B_Q$.
This implies that $x_{\wh P}^1$ is in the same connected
component as $x_Q^1$ and also that $\dist([x_Q^1,x_{\wh P}^1],\partial\Omega)\gtrsim\ell(Q)$, because $b\beta(100B_Q)\leq\ve\ll1$
and they are at the same side of $L_Q$.

Consider now the chain $P=P_1\subset P_2\subset\ldots\subset P_m=\wh P$, so that $P_{i+1}$ is the parent
of $P_i$. Form the curve $\gamma'= \gamma'(x_{\wh P}^1,x_P^1)$ with endpoints $x_{\wh P}^1$ and $x_P^1$ by joining the segments
$[x_{P_i}^1,x_{P_{i+1}}^1]$. Since these segments
satisfy
$$\dist\bigl([x_{P_i}^1,x_{P_{i+1}}^1],\partial\Omega\bigr)\geq c\,\ell(P_i)\geq c\,\ell(P)\geq c\,\ve^{1/4}\,\ell(Q),$$
it is clear that $\dist(\gamma',\partial\Omega)\geq c\,\ve^{1/4}\,\ell(Q)$. 

Next 
we form  a curve $\gamma''= \gamma''(x_Q^1,x_Q^2)$ which joins  $x_Q^1$ to $x_Q^2$
by joining $[x_Q^1,x_{\wh P}^1]$, 
$\gamma'(x_{\wh P}^1,x_P^1)$, and
 $\gamma(x_P^1,x_Q^2)$. It follows easily that this is contained in $90B_Q$ and that $\dist(\gamma'',\partial\Omega)\geq c\,\ve^{1/2}\,\ell(Q)$. However, this is not possible because 
$x_Q^1$ and $x_Q^2$ are in different connected components of $\R^{n+1}\setminus L_Q$ and 
$b\beta(Q)\leq\ve\ll\ve^{1/2}$ (since we assume
$\ve\ll1$).

\vv
\noi {\bf 2.}
Suppose now that $Q\in \wt\eend(R')
$. The arguments are quite similar to the ones above. In this case,
the cubes from
$\mathcal J_2(Q)$ have side length at least $\theta\,\ell(Q)$ and thus 
 at least one of the cubes 
from $\mathcal J_1(P)$ has side length at least $c\,\theta\,\ell(Q)$, which implies that
$\ell(P)\geq c'\,\theta\,\ell(Q)$.

Now there exists a curve $\gamma=\gamma(x_P^1,x_Q^2)$ that joints $x_P^1$ and $x_Q^2$
such that $\dist(\gamma,\partial\Omega)\geq c\,\theta^2\,\ell(Q)$ because all the cubes from 
$\mathcal J_2(Q)$ have side length at least $\theta\,\ell(Q)$, and the ones from $\mathcal J_1(P)$ 
have side length $\theta\,\ell(P)\geq c\,\theta^2\,\ell(Q)$. 

We consider again cubes $\wh P$ and $P_1,\ldots,P_m$ defined exactly as above. By the same reasoning as above, $\dist([x_Q^1,x_{\wh P}^1],\partial\Omega)\gtrsim\ell(Q)$. We also define the curve $\gamma'
=\gamma'(x_{\wh P}^1,x_P^1)$ which 
joins  $x_{\wh P}^1$ to $x_P^1$ in the same way. In the present case we have
$$\dist(\gamma',\partial\Omega)\gtrsim \ell(P)\geq c\,\theta\,\ell(Q).$$
Again construct a curve $\gamma''= \gamma''(x_Q^1,x_Q^2)$ which joins  $x_Q^1$ to $x_Q^2$
by gathering $[x_Q^1,x_{\wh P}^1]$, 
$\gamma'(x_{\wh P}^1,x_P^1)$, and
 $\gamma(x_P^1,x_Q^2)$.  This is contained in $CB_Q$ (for some $C>1$ possibly depending on $\gamma$) and satisfies  $\dist(\gamma'',\partial\Omega)\geq c\,\theta^2\,\ell(Q)$. From this fact we deduce that $x_Q^1$ and $x_Q^2$
can be joined by $C(\theta)$-good Harnack chain. Taking $\Gamma$ big enough (depending on $C(\theta)$),
this implies that the big corkscrews for $Q$ can be joined by a $(c_6\Gamma)$-good Harnack chain, which contradicts
Lemma \ref{lemc6}.

\vv
\noi {\bf 3.}
Finally suppose that $P\in\wt\eend(R')$. We consider the same auxiliary cube $\wh P$ and the same curve
$\gamma=\gamma(x_P^1,x_Q^2)$ satisfying $\dist(\gamma,\partial\Omega)\geq c\,\theta\,\ell(P)$.
 By joining the segments $[x_{P_i}^2,x_{P_{i+1}}^2]$, we construct a curve $\gamma'_2=\gamma'_2(x_{\wh P}^2,x_P^2)$ analogous
to $\gamma'=\gamma'(x_{\wh P}^1,x_P^1)$ from the case 2, so that this joins $x_{\wh P}^2$ to $x_P^2$
and satisfies $\dist(\gamma'_2,\partial\Omega)\gtrsim \ell(P)$.

We construct a curve $\gamma'''$ that joins $x_P^1$ to $x_P^2$ by joining $\gamma(x_P^1,x_Q^2)$,
$[x_Q^2,x_{\wh P}^2]$, and $\gamma'_2(x_{\wh P}^2,x_P^2)$. Again this is contained in 
 $CB_Q$ and it holds $\dist(\gamma''',\partial\Omega)\geq c\,\theta\,\ell(P)$.
This implies that $x_P^1$ and $x_P^2$
can be joined by $C(\theta)$-good Harnack chain. Taking $\Gamma$ big enough,
we deduce the big corkscrews for $P$ can be joined by a $(c_6\Gamma)$-good Harnack chain, which
is a contradiction.
\end{proof}

\vv
By the definition of $V_1$ and $V_2$ it is clear that the properties (a), (b) and (c) in Lemma \ref{lemgeom} hold. So to complete the
proof of the lemma it just remains to prove (d) and (e). 
\vv

\begin{proof}[\bf Proof of Lemma \ref{lemgeom}(d)]
Let $x\in(\partial V_1\cup\partial V_2)\cap 10B_{R'}$. We have to show that there exists some $S\in\wt\eend(R')$
such that $x\in 2B_S$.
To this end we consider $y\in\partial\Omega$ such that $|x-y|=\delta_\Omega(x)$. Since $z_{R'}\in\partial
\Omega$, it follows that $y\in 20B_{R'}$. Let $S\in\wt\eend(R')$ be such that $y\in S$. 
Observe that 
\begin{equation}\label{eqls9}
\ell(S)\leq\frac1{300}\,d_R(y)\leq \frac1{300}\,\bigl(\ell(R') + 20\,r(B_{R'})\bigr) = \frac{81}{300}
\,\ell(R')\leq \frac13\,\ell(R').
\end{equation}

We claim that
$x\in 2B_S$. 
Indeed, if $x\not\in 2B_S$, taking also into account \rf{eqls9}, 
there exists some ancestor $Q$ of $S$ contained in $100B_{R'}$  such that $x\in 2B_Q$ and  
$\delta_\Omega(x)=|x-y|\approx \ell(Q)$. 
From the fact that $S\subsetneq Q\subset 100 B_{R'}$ we deduce that 
$Q\in\wt\tree_\WSBC(R')$. By the construction of the sets $U_i(Q)$, it is immediate to check that the condition that $\delta_\Omega(x)\approx \ell(Q)$ implies that $x\in U_1(Q)\cup U_2(Q)$. Thus $x\in V_1\cup V_2$ and so $x\not \in\partial (V_1\cup V_2) = \partial V_1 \cup \partial V_2$ (for this identity we use that $\dist(V_1,V_2)>0$), which is a contradiction.
\end{proof}
\vv

To show (e), first we need to prove the next result:

\begin{lemma}\label{lembdr1}
For each $i=1,2$, we have
$$\partial V_i \cap 10B_{R'}
\subset \bigcup_{Q\in\wt\eend(R')} \partial U_i(Q).$$
\end{lemma}

\begin{proof}
Clearly, we have 
$$\partial V_i \cap 10B_{R'}
\subset \bigcup_{\substack{P\in\wt\tree_\WSBC(R'):\\P\cap 10B_{R'}\neq\varnothing}} \partial U_i(P) \cup 
\bigcup_{\substack{Q\in\wt\eend(R'):\\Q\cap 10B_{R'}\neq\varnothing}} \partial U_i(Q).$$
So it suffices to show that 
\begin{equation}\label{eqfg66}
\bigcup_{\substack{P\in\wt\tree_\WSBC(R'):\\P\cap 10B_{R'}\neq\varnothing}} \partial U_i(P)\cap 
\partial V_i\cap 10B_{R'}=\varnothing.
\end{equation}

Let $x\in \partial U_i(P)\cap \partial V_i\cap 10B_{R'}$, with $P\in\wt\tree_\WSBC(R')$, $P\cap 10B_{R'}\neq\varnothing$. From the definition
of $U_i(P)$, it follows easily that 
\begin{equation}\label{eqfg67}
\delta_\Omega(x)\gtrsim\ve^{1/4}\ell(P).
\end{equation}
On the other hand, by Lemma \ref{lemgeom}(d), there exists some $Q\in\wt\eend(R')$ such that $x\in 2B_Q$.
By the definition of $U_i(Q)$, since $\theta\ll \ve$, it also follows easily that 
$$\bigl\{y\in 2B_Q:\delta_\Omega(y)>\ve^{1/2}\ell(Q)\bigr\}\subset V_1\cup V_2.$$
Hence,
$\dist(\partial V_i\cap 2B_Q,\partial\Omega)\leq  \ve^{1/2}\,\ell(Q),$
and so 
\begin{equation}\label{eqfg68}
\delta_\Omega(x)\leq  \ve^{1/2}\,\ell(Q).
\end{equation}

 We claim that $\ell(Q)\lesssim\ell(P)$. Indeed, from the fact that $x\in \partial U_i(P)\subset 30B_P$, we infer that
$$30B_P\cap 2B_Q\neq \varnothing.$$
Suppose that $\ell(Q)\geq\ell(P)$. This implies that $B_P\subset33B_Q$. Consider now a cube $S\subset P$ 
belonging to $\wt\eend(R')$. Since $B_S\cap 33B_Q\neq \varnothing$, by Lemma \ref{lem74} (b) we have
$$\ell(Q)\approx \ell(S)\leq \ell(P),$$
which proves our claim. Together with \rf{eqfg67} and \rf{eqfg68}, this yields
$$
\ve^{1/4}\ell(P)\lesssim \delta_\Omega(x)\lesssim \ve^{1/2}\,\ell(Q)\lesssim \ve^{1/2}\,\ell(P),$$
which is a contradiction for $\ve$ small enough.
So there does not exist any $x\in \partial U_i(P)\cap \partial V_i\cap 10B_{R'}$, which proves \rf{eqfg66}.
\end{proof}

\vv

\begin{proof}[\bf Proof of Lemma \ref{lemgeom}(e)]
Let $P\in\wt\eend(R')$ be such that $2B_P\cap10B_{R'}\neq\varnothing$. 
The statement (i) is an immediate consequence of 
Lemma \ref{lemei}. In fact, this lemma implies that any $y\in 2B_P$ such that
$g(p,y)> \gamma\,\ell(P)\,\mu(R_0)^{-1}$
is contained in $U_1(P)\cup U_2(P)$ and thus in $V_1\cup V_2$. In particular, $y\not \in\partial (V_1\cup V_2)=
\partial V_1 \cup\partial V_2$. Thus, if $y\in2B_P\cap\partial V_i$, then
$$g(p,y)\leq\gamma\,\frac{\ell(P)}{\mu(R_0)}.$$
 It is easy to check that this implies the statement (i) in Lemma \ref{lemgeom}(e) (possibly after replacing $\gamma$ by $C\gamma$).

Next we turn our attention to (ii). To this end,
denote by $J_P$ the subfamily of the cubes $Q\in\wt\eend(R')$ such that $30B_Q\cap2B_P\neq\varnothing$.
By Lemma \ref{lembdr1},
\begin{equation}\label{eqsk54}
\partial V_i \cap 2B_P
\subset \bigcup_{Q\in J_P} \partial U_i(Q)\cap 2B_P.
\end{equation}
We will show that 
\begin{equation}\label{eq2eq}
\sum_{I\in\WW_P} \ell(I)^n \lesssim \ell(P)^n
\quad\text{ and }\quad \sum_{I\in\WW_P} \omega^p(B^I) \lesssim \omega^p(CB_P),
\end{equation}
where $\WW_P$ the family of Whitney cubes $I\subset V_1\cup V_2$ such that $1.1\overline I\cap\partial (V_1\cup V_2)\cap 2B_P\neq\varnothing$.
To this end, observe that, by \rf{eqsk54} and the construction of $U_i(Q)$, for each $I\in\WW_P$ there exists some $Q\in J_P$ such that 
$I\subset 30B_Q$ and either $\ell(I)=\theta\ell(Q)$ or $1.1\overline I\cap\partial\wt B_Q\neq\varnothing$. Using the $n$-AD-regularity of $\mu$, it is immediate to check
that for each $Q\in J_P$,
$$\sum_{\substack{I\subset 30B_Q:\\ \ell(I)=\theta\ell(Q)}}\ell(I)^n \lesssim \ell(Q)^n.$$
Also,
$$\sum_{\substack{I\in\WW:\\ 1.1\overline I\cap \partial\wt B_Q\neq\varnothing}}\ell(I)^n\lesssim \sum_{\substack{I\in \WW\\ 1.1\overline I\cap \partial\wt B_Q\neq\varnothing}}\HH^n(2I\cap \partial\wt B_Q) \lesssim \HH^n(\partial\wt B_Q)\lesssim \ell(Q)^n.$$
Since the number of cubes $Q\in J_P$ is uniformly bounded (by Lemma \ref{lem74}(b)) and $\ell(Q)\approx\ell(P)$,
the above inequalities yield the first estimate in \rf{eq2eq}. 

To prove the second one we also distinguish among the two types of cubes $I\in J_P$ above. First, by the
bounded overlap of the balls $B^I$ such that $\ell(I)=\theta\,\ell(Q)$, we get
\begin{equation}\label{eq1p9}
\sum_{\substack{I\subset 30B_Q\\ \ell(I)=\theta\ell(Q)}}\omega^p(B^I) \lesssim \omega^p(CB_P),
\end{equation}
since the balls $B^I$ in the sum are contained $CB_P$ for a suitable universal constant $C>1$.
To deal with the cubes $I\in \WW$ such that $1.1\overline I\cap \partial\wt B_Q\neq\varnothing$ we intend to use the thin boundary property of $\wt B_Q$ in \rf{eqthinbd}. To this end, we write
\begin{align*}
\sum_{\substack{I\in \WW:\\ 1.1\overline I\cap \partial\wt B_Q\neq\varnothing}}\omega^p(B^I) = 
\sum_{k\geq0}\sum_{\substack{I\in \WW:\\ 1.1\overline I\cap \partial\wt B_Q\neq\varnothing\\
\ell(I)=2^{-k}\ell(Q)}}\omega^p(B^I) \lesssim \sum_{k\geq0}\omega^p(\mathcal U_{2^{-k+1}\diam(Q)}(\partial\wt B_Q)),
\end{align*}
where $\mathcal U_d(A)$ stands for the $d$-neighborhood of $A$.
By \rf{eqthinbd} it follows that 
$$\omega^p(\mathcal U_{2^{-k}\ell(Q)}(\partial\wt B_Q))\lesssim 2^{-k}\omega^p(C'B_Q),$$
and thus
$$\sum_{\substack{I\in\WW:\\ 1.1\overline I\cap \partial\wt B_Q\neq\varnothing}}\omega^p(B^I)\lesssim \omega^p(C'B_Q)\lesssim \omega^p(CB_P),$$
for a suitable $C>1$. Together with \rf{eq1p9}, this yields the second inequality in \rf{eq2eq}, which completes the proof of Lemma \ref{lemgeom}(e).
\end{proof}

\vv

% ***************************************************************************

\section{Proof of the Key Lemma}\label{sec7}

We fix $R_0\in\DD_\mu$ and a corkscrew point $p\in\Omega$ as in the preceding sections.
We consider
$R\in\ttt^{(N)}_b$ and we assume $\tree_\WSBC(R)\neq \varnothing$, as in 
Lemma \ref{lemgeom}. We let $R'\in\tree_\WSBC(R)$ be such that
$\ell(R')=2^{-k_0}\ell(R)$, with $k_0=k_0(\gamma)\geq1$ big enough. 
Given $\lambda>0$ and $i=1,2$, we set
\begin{equation}\label{defhi}
\fH_i(R')=\bigl\{Q\in \sss_\WSBC(R)\cap\DD_\mu(R')\cap\fG: g(p,x_{Q}^i)> \lambda\,\ell(Q)\,\mu(R_0)^{-1}\bigr\},
\end{equation}
so that $\sss_\WSBC(R)\cap\DD_\mu(R')\cap\fG = \fH_1(R')\cup \fH_2(R')$.
Here we are assuming that the corkscrews $x_Q^i$ belong to the set $V_i$ from Lemma \ref{lemgeom}
and that $\lambda$ is small enough.
%Recall that, by Lemma \ref{lem5.1},
%$$\omega^p(2Q)\lesssim \frac{\mu(Q)}{\mu(R_0)}\quad\mbox{ for all $Q\in \tree^{(N)}(R)$.}$$

\vv

\begin{lemma}[Baby Key Lemma]\label{keylemma}
Let $p,R_0,R,R'$ be as above.
Given $\lambda>0$, define also $\fH_i(R')$ as above.
%Let $R\in\ttt^{(N)}$ be such that $b\beta(R)\leq\ve$ and $\tree_\WSBC(R)\neq \varnothing$. 
For a given $\tau>0$, suppose that
$$\mu\biggl(\,\bigcup_{Q\in \fH_i(R')}Q\biggr) \geq \tau\,\mu(R').$$
If $\gamma$ is small enough in the definition of $V_i$ in Lemma \ref{lemgeom} (depending on $\tau$ and $\lambda$), then
$$g(p,x_{R'}^i)\geq c(\lambda,\tau)\,\frac{\ell(R')}{\mu(R_0)}.$$
\end{lemma}

Remark that $\Gamma$ depends on $\gamma$ (see Lemma \ref{lemgeom}), and thus the families $\WSBC(\Gamma)$,
$\sss_\WSBC(R)$, $\fH_i(R')$ also depend on $\gamma$. The reader should thing that $\Gamma\to\infty$
as $\gamma\to0$.

A key fact in this lemma is that the constants $\lambda,\tau$ can be taken arbitrarily small, without requiring
$\ve\to0$ as $\lambda\tau\to0$. Instead, the lemma requires $\gamma\to0$, which does not affect the packing condition in Lemma \ref{lempack}.

We denote
$$\bdy(R')=\bigcup_{P\in \wt\eend(R'):2B_P\cap 10B_{R'}\neq\varnothing}\WW_P,$$
with $\WW_P$ as in the Lemma \ref{lemgeom}. That is, 
$\WW_P$ the family of Whitney cubes $I\subset V_1\cup V_2$ such that $1.1\overline I\cap\partial (V_1\cup V_2)\cap 2B_P\neq\varnothing$.
So the family $\bdy(R')$ contains Whitney cubes which intersect the boundaries of $V_1$ or $V_2$ and are close to $10B_{R'}$. 

To prove Lemma \ref{keylemma}, first we need the following auxiliary result.

\begin{lemma}\label{lemgg}
Let $p,R_0,R,R'$ be as above and, for $i=1$ or $2$, let $Q\in\fH_i(R')$.
Let $V_i$ be as in Lemma \ref{lemgeom} and let $q\in \Omega$ be a corkscrew point for $Q$ which belongs to $V_i$.
Denote $r=2\ell(R')$ and for $\delta\in (0,1/100)$ set
$$A_r^\delta = \bigl\{x\in A(q,r,2r)\cap\Omega:\delta_\Omega (x)>\delta\,r\bigr\}.$$
Then we have
\begin{align*}
g(p,q)&\lesssim \frac{1}{r} \sup_{y\in A_r^\delta\cap V_i} \frac{g(p,y)}{\delta_\Omega (y)}\,\,\int_{A_r^\delta}  g(q,x)\,dx \\
&\quad + \frac{\delta^{\alpha/2}}{r^{n+3}} \int_{A(q,r,2r)}g(p,x)\,dx\,
\int_{A(q,r,2r)} g(q,x)\,dx\\
&\quad + \sum_{I\in\bdy(R')}\frac1{\ell(I)}\int_{2I} \bigl| g(p,x)\,\nabla  g(q,x)- \nabla g(p,x)\, g(q,x)\bigr|\,dx.
\end{align*}
\end{lemma}

Note that the fact that $q$ is a corkscrew for $Q$ contained in $V_i$ implies that
$\dist(q,\partial V_i)\approx \ell(Q)$, by the construction of the sets $V_i$ in Lemma
\ref{lemgeom}.

\begin{proof}
 We fix $i=1$, for definiteness. Recall that 
$V_1=\bigcup_{I\in \WW_1} 1.1\mathring I$. 
For each $I\in\WW_1$, consider a smooth function $\eta_I$ such that $\chi_{0.9I}\leq \eta_I\leq \chi_{1.09I}$
with $\|\nabla\eta_I\|_\infty\lesssim \ell(I)^{-1}$ and
$$\eta:=\sum_{I\in\WW_1} \eta_I\equiv 1\quad \text{ on $V_1\cap 10B_{R'}\setminus \bigcup_{I\in\bdy(R')} 2I$}.$$
It follows that $\supp\eta\subset V_1$ and so
$\supp\eta \cap V_2=\varnothing,$
and also
$$\supp(\nabla\eta)\cap10B_{R'}\subset \bigcup_{I\in\bdy(R')} 2I.$$

 Let $\vphi_0$ be a smooth function such that $\chi_{B(q,1.2r)}\leq \vphi_0\leq \chi_{B(q,1.8r)}$, with $\|\nabla\vphi_0\|_\infty
\lesssim 1/r$. Then we set
$$\vphi = \eta\,\vphi_0.$$
So $\vphi$
is smooth, and it satisfies
$$\supp\nabla\vphi\subset \bigl(A(q,r,2r)\cap V_1\bigr)\cup \bigcup_{I\in \bdy(R')} 2I.$$
Observe that, in a sense, 
$\vphi$ is a smooth version of the function $\chi_{B(q,r)\cap V_1}$.

Since $g(p,q) = g(p,q)\,\vphi(q)$ and $g(p,\cdot)\,\vphi$ is a continuous function from $W_0^{1,2}(\Omega)$, we have
\begin{align*}
g(p,q)   & = \int_{\Omega} \nabla (g(p,\cdot)\,\vphi)(x)\,\nabla  g(q,x)\,dx\\
& = \int_{\Omega} g(p,x)\,\nabla\vphi(x)\,\nabla  g(q,x)\,dx
+ \int_{\Omega} \vphi(x)\,\nabla g(p,x)\,\nabla  g(q,x)\,dx\\
&=: I_1+ I_2.
\end{align*}

First we estimate $I_2$. For $\ve$ with $0<\ve<1/10$, we consider a smooth function $\vphi_\ve$ such that $\chi_{B(q,\ve \delta_\Omega (q))}\leq \vphi_\ve\leq \chi_{B(q,2\ve \delta_\Omega (q))}$, with $\|\nabla\vphi_\ve\|_\infty\lesssim 1/(\ve \delta_\Omega (q))$.
Since $\vphi_\ve\,\vphi = \vphi_\ve$, we have
$$I_2= \int_{\Omega} \vphi_\ve(x)\,\nabla g(p,x)\,\nabla  g(q,x)\,dx + \int_{\Omega} \vphi(x)(1-\vphi_\ve(x))\,\nabla g(p,x)\,\nabla  g(q,x)\,dx =: I_{2,a} + I_{2,b}.$$
To deal with $I_{2,a}$ we use the fact that for $x\in B(q,2\ve \delta_\Omega (q))$ we have
$$|\nabla  g(q,x)|\lesssim \frac1{|x-q|^n} \quad \text{ and }\quad |\nabla g(p,x)|\lesssim \frac{g(p,q)}{\delta_\Omega(q)}.$$
Then we get
$$|I_{2,a}|\lesssim \frac{g(p,q)}{\delta_\Omega(q)}\int_{B(q,2\ve \delta_\Omega(q))} \frac1{|x-q|^n}\,dx
\lesssim \frac{g(p,q)}{\delta_\Omega(q)}\,\ve\,\delta_\Omega(q) = \ve \,g(p,q).
$$

Let us turn our attention to $I_{2,b}$. We denote $\psi = \vphi(1-\vphi_\ve)$. Integrating by parts, we get
$$I_{2,b} = \int  \nabla g(p,x)\,\nabla (\psi\, g(q,\cdot))(x)\,dx - \int \nabla g(p,x)\,\nabla \psi(x)\,\, g(q,x)\,dx.$$
Observe now that the first integral vanishes because $\psi\, g(q,\cdot)\in W_0^{1,2}(\Omega)\cap C(\overline
\Omega)$ and vanishes at $\partial
\Omega$ and at $p$. Hence, since $\nabla\psi = \nabla \vphi - \nabla\vphi_\ve$, we derive
$$I_{2,b} = - \int \nabla g(p,x)\,\nabla \vphi(x)\,\, g(q,x)\,dx
+ \int \nabla g(p,x)\,\nabla \vphi_\ve(x)\,\, g(q,x)\,dx = I_3+ I_4.$$

To estimate $I_4$ we take into account that
$|\nabla \vphi_\ve|\lesssim \chi_{A(q,\ve \delta_\Omega(q),2\ve \delta_\Omega(q))}/ (\ve \delta_\Omega(q))$, and then we derive
$$|I_4|\lesssim  \frac1{\ve\,\delta_\Omega(q)}\int_{A(q,\ve \delta_\Omega(q),2\ve \delta_\Omega(q))} |\nabla g(p,x)|\, g(q,x)\,dx.$$
Using now that, for $x$ in the domain of integration,
$$ g(q,x)\lesssim \frac1{(\ve\,\delta_\Omega (q))^{n-1}} \quad \text{ and }\quad |\nabla g(p,x)|\lesssim \frac{g(p,q)}{\delta_\Omega (q)},$$
we obtain
$$|I_4|\lesssim  \frac1{\ve\,\delta_\Omega (q)}\,\frac1{(\ve\,\delta_\Omega (q))^{n-1}}\,\frac{g(p,q)}{\delta_\Omega (q)}\,(\ve \,\delta_\Omega (q))^{n+1}\lesssim \ve\,g(p,q).$$

From the above estimates we infer that
$$g(p,q) \leq |I_1 + I_3| + c\,\ve\,g(p,q).$$
Since neither $I_1$ nor $I_3$ depend on $\ve$, letting $\ve\to0$ we get
\begin{align*}
g(p,q)   & \leq |I_1 + I_3|\\
& \leq\left|\int g(p,x)\,\nabla\vphi(x)\,\nabla  g(q,x)\,dx
 - \!\int \nabla g(p,x)\,\nabla \vphi(x)\, g(q,x)\,dx\right|\\
 &\leq \int |\nabla\vphi(x)|\bigl| g(p,x)\,\nabla  g(q,x)- \nabla g(p,x)\, g(q,x)\bigr|\,dx.
\end{align*} 
We denote 
$$\wt F= \bigcup_{I\in\bdy(R')} 2I,$$
$$ \wt A_r^\delta = \bigl\{x\in A(q,1.2r,1.8r)\cap V_1\setminus \wt F:\delta_\Omega (x)>\delta\,r\bigr\},$$
and
$$\wt A_{r,\delta} = \bigl\{x\in A(q,1.2,1.8r)\cap V_1\setminus \wt F:\delta_\Omega (x)\leq\delta\,r\bigr\}.$$
Next we split the last integral as follows:
\begin{align}\label{eqj1j2}
g(p,q) &\leq \int_{\wt A_r^\delta} |\nabla\vphi(x)|\,\bigl| g(p,x)\,\nabla  g(q,x)- \nabla g(p,x)\, g(q,x)\bigr|\,dx\\
&\quad +
\int_{\wt A_{r,\delta}} |\nabla\vphi(x)|\,\bigl| g(p,x)\,\nabla  g(q,x)- \nabla g(p,x)\, g(q,x)\bigr|\,dx\nonumber
\\ 
&\quad+ \int_{\wt F} |\nabla\vphi(x)|\,\bigl| g(p,x)\,\nabla  g(q,x)- \nabla g(p,x)\, g(q,x)\bigr|\,dx\nonumber\\
&=: J_1 + J_2+ J_3.
\nonumber
\end{align}

Concerning $J_1$, we have
$$|\nabla  g(p,x)|\lesssim \frac{ g(p,x)}{\delta_\Omega (x)} \quad\text{ and }\quad |\nabla g(q,x)|\lesssim \frac{g(q,x)}{\delta_\Omega (x)}\quad
\text{ for all $x\in \wt A_r^\delta$.}$$
Thus, using also that $|\nabla\vphi|\lesssim 1/r$ outside $\wt F$,
\begin{equation}\label{eqj1*}
J_1 \lesssim \frac{1}{r} \sup_{x\in A_r^\delta\cap V_1} \frac{g(p,x)}{\delta_\Omega (x)}\,\,\int_{A_r^\delta}  g(q,x)\,dx.
\end{equation}

Regarding $J_2$, using  Cauchy-Schwarz, we get
\begin{align}\label{eqcasw1}
J_2&\lesssim \frac1r\int_{\wt A_{r,\delta}} \bigl| g(p,x)\,\nabla  g(q,x)- \nabla g(p,x)\, g(q,x)\bigr|\,dx\\
& \leq \frac1r
\left(\int_{\wt A_{r,\delta}} g(p,x)^2\,dx\right)^{1/2} \,\left(\int_{\wt A_{r,\delta}} |\nabla  g(q,x)|^2\,dx\right)^{1/2}
\nonumber
\\ &\,+\frac{1}{r}
\left(\int_{\wt A_{r,\delta}} |\nabla g(p,x)|^2\,dx\right)^{1/2} \left(\int_{\wt A_{r,\delta}} g(q,x)^2\,dx\right)^{1/2}.
\nonumber
\end{align}
To estimate the integral $\int_{\wt A_{r,\delta}} g(p,x)^2\,dx$, we take into account that, for all $x\in \wt A_{r,\delta}$,
$$g(p,x)\lesssim \delta^\alpha\;\avint_{A(q,r,2r)}g(p,y)\,dy.$$
Then we deduce
\begin{equation*}%\label{eqgpx}
\int_{\wt A_{r,\delta}} g(p,x)^2\,dx \lesssim \frac{\delta^\alpha}{r^{n+1}} \left(\int_{A(q,r,2r)}g(p,x)\,dx\right)^2.
\end{equation*}

Next we estimate the integral $\int_{\wt A_{r,\delta}} |\nabla  g(q,x)|^2\,dx$.
By covering $\wt A_{r,\delta}$ by a finite family of balls of radius $r/100$ and applying Cacciopoli's inequality to each one, it follows that
$$\int_{\wt A_{r,\delta}} |\nabla  g(q,x)|^2\,dx \lesssim \frac1{r^2}\int_{A(q,1.1r,1.9r)}  g(q,x)^2\,dx.$$
Since
$$ g(q,x)\lesssim \;\avint_{A(q,r,2r)} g(q,y)\,dy\quad \mbox{ for all $x\in A(q,1.1r,1.9r)$,}$$ we get
%$$\int_{A_{r,\delta}}  g(q,x)^2\,dx \lesssim \frac{1}{r^{n+1}} \left(\int_{A(q,1.8r,2.2r)} g(q,x)\,dx\right)^2.$$
%Thus,
$$\int_{\wt A_{r,\delta}} |\nabla  g(q,x)|^2\,dx \lesssim
\frac1{r^2}\int_{A(q,1.1r,1.9r)}  g(q,x)^2\,dx\lesssim
\frac{1}{r^{n+3}}
 \left(\int_{A(q,r,2r)} g(q,x)\,dx\right)^2.$$
So we obtain
\begin{multline*}
\left(\int_{\wt A_{r,\delta}} \!\!g(p,x)^2\,dx\right)^{1/2} \,\left(\int_{\wt A_{r,\delta}} |\nabla  g(q,x)|^2\,dx\right)^{1/2}
\\ \lesssim \frac{\delta^{\alpha/2}}{r^{n+2}} \int_{A(q,r,2r)}g(p,x)\,dx\,
\int_{A(q,r,2r)} g(q,x)\,dx.
\end{multline*}
By interchanging, $p$ and $q$, it is immediate to check that an analogous estimate holds for the second summand on the
right hand side of \rf{eqcasw1}.
Thus we get
\begin{equation}\label{eqfi8}
J_2
\lesssim \frac{\delta^{\alpha/2}}{r^{n+3}} \int_{A(q,r,2r)}g(p,x)\,dx\,
\int_{A(q,r,2r)} g(q,x)\,dx.
\end{equation}

Concerning $J_3$, we just take into account that $|\nabla\vphi|\lesssim 1/\ell(I)$ in $2I$, and then we obtain
$$J_3\lesssim \sum_{I\in\bdy(R')} \frac1{\ell(I)}\int_{2I} \bigl| g(p,x)\,\nabla  g(q,x)- \nabla g(p,x)\, g(q,x)\bigr|\,dx.$$
Together with \rf{eqj1j2}, \rf{eqj1*}, and \rf{eqfi8}, this yields the lemma.
\end{proof}
\vv

\begin{proof}[\bf Proof of Lemma \ref{keylemma}]
 We fix $i=1$, for definiteness.
By a Vitali type covering theorem, there exists a subfamily $\wt\fH_1(R')\subset \fH_1(R')$ such that the balls
$\{8B_Q\}_{Q\in  \wt\fH_1(R')}$ are disjoint and 
$$\sum_{Q\in \fH_1(R')}\mu(Q) \lesssim \sum_{Q\in \wt\fH_1(R')}\mu(Q).$$

 By Lemma \ref{lemgg}, for each $Q\in\wt\fH_1(R')$  we have
\begin{align*}
g(p,x_Q^1)&\lesssim \frac{1}{r} \sup_{y\in 2B_{R'}\cap V_1:\delta_\Omega(y)\geq \delta\ell(R')} \frac{g(p,y)}{\delta_\Omega (y)}\,\,\int_{A(x_Q^1,r,2r)}  g(x_Q^1,x)\,dx \\
&\quad + \frac{\delta^{\alpha/2}}{r^{n+3}} \int_{A(x_Q^1,r,2r)}g(p,x)\,dx\,
\int_{A(x_Q^1,r,2r)} g(x_Q^1,x)\,dx\\
&\quad + 
\sum_{I\in\bdy(R')}\frac1{\ell(I)}\int_{2I}
 \bigl| g(p,x)\,\nabla  g(x_Q^1,x)- \nabla g(p,x)\, g(x_Q^1,x)\bigr|\,dx\\
& =: I_1(Q) + I_2(Q) + I_3(Q),
\end{align*}
with $r=2\ell(R')$.
Since $g(p,x_{Q}^1)> \lambda\,\ell(Q)/\mu(R_0)$, we derive
\begin{equation}\label{eqi123}
\lambda\tau\,\mu(R')\lesssim \lambda\!\sum_{Q\in \wt\fH_1(R')}\!\mu(Q) \lesssim 
\sum_{Q\in \wt\fH_1(R')} \!g(p,x_Q^1)\,\ell(Q)^{n-1}\,\mu(R_0)\lesssim \sum_{j=1}^3 \sum_{Q\in \wt\fH_1(R')}  I_j(Q)\,\ell(Q)^{n-1}\,\mu(R_0).
\end{equation}

\vv
\subsection*{Estimate of $\sum_{Q\in \wt\fH_1(R')}  I_1(Q)\,\ell(Q)^{n-1}$}
We have
$$\sum_{Q\in \wt\fH_1(R')}  I_1(Q)\,\ell(Q)^{n-1}\leq 
\frac{1}{r} \sup_{y\in 2B_{R'}\cap V_1:\delta_\Omega(y)\geq \delta\ell(R')} \frac{g(p,y)}{\delta_\Omega (y)}\,\sum_{Q\in \wt\fH_1(R')}\int_{A(x_Q^1,r,2r)}  g(x_Q^1,x)\,dx\,\ell(Q)^{n-1}.$$
Note now that
$$\sum_{Q\in \wt\fH_1(R')}\int_{A(x_Q^1,r,2r)}  g(x_Q^1,x)\,dx \,\ell(Q)^{n-1}\lesssim \int_{2B_{R'}} \sum_{Q\in \wt\fH_1(R')} \omega^x(4Q)\,dx \leq \int_{2B_{R'}} 1\,dx\lesssim \ell(R')^{n+1}.$$
Since $r\approx\ell(R')$, we derive
$$\sum_{Q\in \wt\fH_1(R')}  I_1(Q)\,\ell(Q)^{n-1}\lesssim \sup_{y\in 2B_{R'}\cap V_1:\delta_\Omega(y)\geq \delta\ell(R')} \frac{g(p,y)}{\delta_\Omega (y)} \,\mu(R').$$

\vv
\subsection*{Estimate of $\sum_{Q\in \wt\fH_1(R')}  I_2(Q)\,\ell(Q)^{n-1}$}
First we estimate $\int_{A(x_Q^1,r,2r)}g(p,x)\,dx$ by applying Lemma \ref{lem1'}:
\begin{align*}
\int_{A(x_Q^1,r,2r)}g(p,x)\,dx\leq \int_{2B_{R'}}g(p,x)\,dx\lesssim 
\ell(R')^{n+1}\,\frac{\omega^p(8B_{R'})}{\ell(R')^{n-1}} \lesssim \ell(R')^2\,\frac{\mu(R')}{\mu(R_0)}\approx
\frac{r^{n+2}}{\mu(R_0)}.
\end{align*}
So we have
\begin{align*}
\sum_{Q\in \wt\fH_1(R')}  I_2(Q)\,\ell(Q)^{n-1} & \lesssim
\frac{\delta^{\alpha/2}}{r\,\mu(R_0)} \sum_{Q\in \wt\fH_1(R')}\int_{A(x_Q^1,r,2r)} g(x_Q^1,x)\,dx\,\ell(Q)^{n-1}\\
 & \lesssim\frac{\delta^{\alpha/2}}{r\,\mu(R_0)}\int_{2B_{R'}} \sum_{Q\in \wt\fH_1(R')} \omega^x(4Q)\,dx\\
 & \lesssim
\frac{\delta^{\alpha/2}}{r\,\mu(R_0)} \int_{2B_{R'}} 1\,dx \lesssim
\frac{\delta^{\alpha/2}\,\mu(R')}{\mu(R_0)}.
\end{align*}

\vv
\subsection*{Estimate of $\sum_{Q\in \wt\fH_1(R')}  I_3(Q)\,\ell(Q)^{n-1}$}
Note first that, for each $I\in\bdy(R')$, since $x_Q^1\not \in 4I$, using the subharmonicity of $g(p,\cdot)$ and
$g(x_Q^1,\cdot)$ in $4I$, and Caccioppoli's inequality,
\begin{align*}
\frac1{\ell(I)}\int_{2I}\bigl| g(p,x)\,\nabla  g(x_Q^1,x)\bigr|\,dx&\lesssim
\frac1{\ell(I)}\sup_{x\in 2I} g(p,x) \int_{2I} |\nabla  g(x_Q^1,x)|\,dx\\
&\lesssim \ell(I)^{n-1}\,m_{4I}g(p,\cdot) \,\, m_{4I}g(x_Q^1,\cdot).
\end{align*}
By very similar estimates, we also get
$$\frac1{\ell(I)}\int_{2I}\bigl| \nabla g(p,x)\, g(x_Q^1,x)\bigr|\,dx\lesssim\ell(I)^{n-1}\,m_{4I}g(p,\cdot) \,\, m_{4I}g(x_Q^1,\cdot).$$
Recall now that, by Lemma \ref{lemgeom}(e)(i),
$$m_{4I} g(p,\cdot)\leq \gamma\,\frac{\ell(P)}{\mu(R_0)}\quad \mbox{ for each $I\in\WW_P$,
with $P\in \wt\eend(R')$ such that $2B_P\cap10B_{R'}\neq\varnothing$.}$$ 

We distinguish two types of Whitney cubes $I\in\bdy(R')$. We write $I\in T_1$ if $\ell(I)\geq \gamma^{1/2}\ell(P)$ for some $P$ such that $I\in\WW_P$ and $2B_P\cap 10B_{R'}\neq\varnothing$, and we write $I
\in T_2$ otherwise (there may exist more than one $P$ such that
$I\in\WW_P$, but if $\WW_P\cap\WW_{P'}\neq\varnothing$, then $\ell(P)\approx\ell(P')$).
So we split
\begin{align}\label{eqs1s2}
\sum_{Q\in \wt\fH_1(R')}  I_3(Q)\,\ell(Q)^{n-1} & 
\leq 
\sum_{Q\in \wt\fH_1(R')} \sum_{I\in\bdy(R')}\ell(I)^{n-1}\,m_{4I}g(p,\cdot) \,\, m_{4I}g(x_Q^1,\cdot)\,\ell(Q)^{n-1}\nonumber\\
&= \sum_{Q\in \wt\fH_1(R')} \sum_{I\in T_1}\ldots + \sum_{Q\in \wt\fH_1(R')} \sum_{I\in T_2}\ldots =: S_1+S_2.
\end{align}

Concerning the sum $S_1$ we have
\begin{align*}
S_1 &
\lesssim  \gamma
\sum_{Q\in \wt\fH_1(R')} \sum_{\substack{P\in \wt\eend(R'):\\2B_P\cap10B_{R'}\neq\varnothing}} \sum_{I\in \WW_P\cap T_1}\frac{\ell(P)}{\mu(R_0)}\,
\ell(I)^{n-1}\, m_{4I}g(x_Q^1,\cdot)\,\ell(Q)^{n-1}\\
&\lesssim\gamma^{1/2}
\sum_{Q\in \wt\fH_1(R')} \sum_{\substack{P\in \wt\eend(R'):\\2B_P\cap10B_{R'}\neq\varnothing}} \sum_{I\in \WW_P}\frac{\ell(I)^{n}}{\mu(R_0)}\,
 m_{4I}g(x_Q^1,\cdot)\,\ell(Q)^{n-1}
\end{align*}
Next we take into account that
$$\ell(Q)^{n-1}\, m_{4I}g(x_Q^1,\cdot)\lesssim \omega^{x_I}(4Q),$$ 
where $x_I$ stands for the center of $I$. Then we derive
$$S_1
\lesssim \gamma^{1/2}
\sum_{Q\in \wt\fH_1(R')} \sum_{\substack{P\in \wt\eend(R'):\\2B_P\cap10B_{R'}\neq\varnothing}} \sum_{I\in \WW_P}
\omega^{x_I}(4Q)\,\frac{\ell(I)^n}{\mu(R_0)}.$$
Since $\sum_{Q\in \wt\fH_1(R')}\omega^{x_I}(4Q)\lesssim 1$ for each $I$, we get
$$S_1
\lesssim \gamma^{1/2}
 \sum_{\substack{P\in \wt\eend(R'):\\2B_P\cap10B_{R'}\neq\varnothing}} \sum_{I\in \WW_P}
\frac{\ell(I)^n}{\mu(R_0)}.$$
By Lemma \ref{lemgeom}(e)(ii), we have $\sum_{I\in \WW_P} \ell(I)^n\lesssim\ell(P)^n$, and so we deduce
$$S_1\lesssim \gamma^{1/2}
 \sum_{\substack{P\in \wt\eend(R'):\\2B_P\cap10B_{R'}\neq\varnothing}} 
\frac{\mu(P)}{\mu(R_0)}\lesssim \gamma^{1/2}\frac{\mu(R')}{\mu(R_0)}.$$

\vv

Next we turn our attention to the sum $S_2$ in \rf{eqs1s2}. Recall that
\begin{align*}
S_2& =\sum_{Q\in \wt\fH_1(R')} 
 \sum_{I\in  T_2}
\ell(I)^{n-1}\,m_{4I}g(p,\cdot) \,\, m_{4I}g(x_Q^1,\cdot)\,\ell(Q)^{n-1}.
\end{align*}
Let us remark that we assume the condition that $I\in\WW_P$ for some 
$2P\in \wt\eend(R')$ such that $2B_P\cap10B_{R'}\neq\varnothing$ to be part of the definition of $I\in T_2$.
Using the estimate
$m_{4I}g(p,\cdot) \lesssim \omega^p(B^I)\,\ell(I)^{1-n}$,
we derive
\begin{align*}
S_2& \lesssim\sum_{Q\in \wt\fH_1(R')} 
 \sum_{I\in T_2}
\omega^p(B^I) \, m_{4I}g(x_Q^1,\cdot)\,\ell(Q)^{n-1}\\
& = 
\sum_{Q\in \wt\fH_1(R')} 
 \sum_{I\in T_2:20I\cap 20 B_Q\neq\varnothing}
\ldots +
\sum_{Q\in \wt\fH_1(R')} 
 \sum_{I\in T_2:20I\cap 20 B_Q =\varnothing}
\ldots =: A+ B.
\end{align*}
To estimate the term $A$ we take into account that if $20I\cap 20B_Q\neq\varnothing$ and $I\in\WW_P$, then
$\ell(P)\lesssim \ell(Q)$ and thus $\ell(I)\lesssim\gamma^{1/2}\,\ell(Q)$ because $I\in T_2$.
As a consequence, $I\subset 21B_Q$ and also, by the H\"older 
continuity of $g(x_Q^1,\cdot)$, if we let $B$ be a ball concentric with $B^I$ with radius comparable to $\ell(Q)$ and such
that $\dist(x_Q^1,B)\approx \ell(Q)$, 
we obtain
$$m_{2B^I}g(x_Q^1,\cdot)\lesssim \biggl(\frac{r(B^I)}{r(B)}\biggr)^\alpha\,m_B g(x_Q^1,\cdot)\lesssim \gamma^{\alpha/2}\,\frac1{\ell(Q)^{n-1}},$$
where $\alpha>0$ is the exponent of H\"older continuity.
Hence,
$$A\lesssim \gamma^{\alpha/2} \sum_{Q\in \wt\fH_1(R')} \sum_{\substack{P\in \wt\eend(R'):\\2B_P\cap10B_{R'}\neq\varnothing\\
20B_P\cap20 B_Q\neq\varnothing}} 
\sum_{I\in \WW_P\cap T_2}
\omega^p(B^I).$$
By Lemma \ref{lemgeom}(e)(ii), we have $\sum_{I\in \WW_P}
\omega^p(B^I)\lesssim \omega^p(CB_P)$,
and using also that, for $P$ as above, $CB_P\subset C'B_Q$ for some absolute constant $C'$, we obtain
$$A\lesssim  \gamma^{\alpha/2} \sum_{Q\in \wt\fH_1(R')} 
\omega^p(C' B_Q)\lesssim \gamma^{\alpha/2} \sum_{Q\in \wt\fH_1(R')} \frac{\mu(Q)}{\mu(R_0)}\lesssim \gamma^{\alpha/2} \frac{\mu(R')}{\mu(R_0)}
.$$

Finally, we turn our attention to the term $B$. We have
\begin{align*}
B&= \sum_{Q\in \wt\fH_1(R')} 
 \sum_{I\in T_2:20I\cap 20 B_Q =\varnothing} \omega^p(B^I) \, m_{4I}g(x_Q^1,\cdot)\,\ell(Q)^{n-1}\\
& = \sum_{I\in T_2} \omega^p(B^I)\; \avint_{4I} \sum_{Q\in \wt\fH_1(R'):20I\cap 20B_Q = \varnothing}
g(x_Q^1,x)\,\ell(Q)^{n-1}\,dx\\
&\lesssim \sum_{I\in T_2} \omega^p(B^I)\; \avint_{4I} \sum_{Q\in \wt\fH_1(R'):20I\cap 20B_Q = \varnothing}\omega^x(8B_Q)\,dx.
\end{align*}
We claim now that, in the last sum, if $20I\cap 20B_Q=\varnothing$, then $\dist(I,8B_Q)\geq c\,\gamma^{-1/2}\,\ell(I)$.
To check this, take $P\in\wt\eend(R')$ such that $I\in\WW_P$. Then note that
\begin{align*}
\ell(P)\leq \frac1{300}\,d_R(z_P)&\leq\frac1{300}\,\bigl(\dist(z_P,Q) + \ell(Q)\bigr)\\
&\leq \frac1{300}\,\bigl(\dist(z_P,I)+
\diam(I)+\dist(I,8B_Q) + C\ell(Q)\bigr).
\end{align*}
Using that $I\cap 2B_P\neq\varnothing$, $\diam(I)\leq C\gamma^{1/2}\ell(P)\ll\ell(P)$, and $\ell(Q)\leq \dist(I,8B_Q)$, we get
$$\ell(P) \leq\frac1{300}\,\bigl(\dist(I,8B_Q)+3r(B_P)+C\,\ell(Q)\bigr)
\leq C\,\dist(I,8B_Q) + \frac{12}{300}\,\ell(P),$$
which implies  that
$$\ell(I) \leq C\gamma^{1/2}\,\ell(P)\leq C\,\gamma^{1/2}\,\dist(I,8B_Q),$$
and yields our claim.

Taking into account that the balls $\{8B_Q\}_{Q\in  \wt\fH_1(R')}$ are disjoint and the H\"older continuity of $\omega^{(\cdot)}(\partial\Omega\setminus c\gamma^{-1/2} I)$, for all $x\in 4I$
we get
$$\sum_{Q\in \wt\fH_1(R'):20I\cap 20B_Q = \varnothing}\omega^x(8B_Q)\lesssim \omega^x(\partial\Omega\setminus c\gamma^{-1/2} I)
\lesssim\gamma^{\alpha/2}.$$
Thus,
$$B
\lesssim \gamma^{\alpha/2} \sum_{I\in T_2}\omega^p(B^I) \leq \gamma^{\alpha/2}\sum_{\substack{P\in \wt\eend(R'):\\2B_P\cap10B_{R'}\neq\varnothing}} 
\sum_{I\in \WW_P\cap T_2}\omega^p(B^I)
.$$
Recalling again that $\sum_{I\in \WW_P}
\omega^p(B^I)\lesssim \omega^p(CB_P)$, we deduce
$$B
\lesssim \gamma^{\alpha/2} \sum_{\substack{P\in \wt\eend(R'):\\2B_P\cap10B_{R'}\neq\varnothing}} \omega^p(CB_P)
\lesssim\gamma^{\alpha/2}\sum_{\substack{P\in \wt\eend(R'):\\2B_P\cap10B_{R'}\neq\varnothing}}\frac{\mu(P)}{\mu(R_0)}\lesssim \gamma^{\alpha/2}\,\frac{\mu(R')}{\mu(R_0)}
.$$
Remark that for the second inequality we took into account that $P$ is contained in a cube of the form $22P'$ with $P'\in\tree_\WSBC(R)$ and $\ell(P')\approx\ell(P)$, by Lemma \ref{lem74}. This implies that $\omega^p(CB_P)\leq
\omega^p(C'B_{P'})\lesssim \mu(P')\,\mu(R_0)^{-1}\lesssim  \mu(P)\,\mu(R_0)^{-1}$.

\vv

Gathering the estimates above and recalling \rf{eqi123}, we deduce
$$\lambda\tau\,\mu(R')\lesssim \sup_{y\in 2B_{R'}\cap V_1:\delta_\Omega(y)\geq \delta\ell(R')} \frac{g(p,y)}{\delta_\Omega (y)} \,\mu(R')\,\mu(R_0)+
\delta^{\alpha/2}\,\mu(R') + \gamma^{\alpha/2}\,\mu(R').$$
So, if $\delta$ and $\gamma$ are small enough (depending on $\lambda,\tau$), we infer that
$$\lambda\,\tau\,\mu(R')\lesssim \sup_{y\in 2B_{R'}\cap V_1:\delta_\Omega(y)\geq \delta\ell(R')} \frac{g(p,y)}{\delta_\Omega (y)} \,\mu(R')\,\mu(R_0).$$
That is, there exists some $y_0\in 2B_{R'}\cap V_1$ with $\delta_\Omega(y_0)\geq \delta\,\ell(R')$ such that
$$ \frac{g(p,y_0)}{\delta_\Omega (y)}\gtrsim \frac{\lambda\tau}{\mu(R_0)},$$ 
with $\delta$ depending on $\lambda,\tau$. Since $x_{R'}^1$ and $y_0$ can be joined by a $C$-good 
Harnack chain (for some $C$ depending on $\delta$ and $\gamma$, and thus on $\lambda,\tau$), we deduce that
$$ \frac{g(p,x_{R'}^1)}{\ell(R')} \gtrsim \frac{c(\lambda,\tau)}{\mu(R_0)},$$
as wished. 
\end{proof}

\vv
\begin{lemma}\label{keylemma2}
Let $\eta\in (0,1)$ and $\lambda>0$. Choose $\gamma=\gamma(\lambda,\tau)$ small enough as in Lemma \ref{keylemma} with $\tau=\eta/2$.
Assume that the family $\WSBC(\Gamma)$ is defined by choosing $\Gamma$ big enough depending on $\gamma$ (and thus on $\lambda$ and $\eta$) as in Lemma
\ref{lemgeom}.
Let $R\in\ttt^{(N)}_b$ and suppose that $\tree_\WSBC(R)\neq \varnothing$.
Then, there exists an exceptional family $\mathsf{Ex}_\WSBC(R)\subset \sss_\WSBC(R)\cap\fG$ 
satisfying
$$\sum_{P\in\mathsf{Ex}_\WSBC(R)}\mu(P)\leq \eta\,\mu(R)$$
such that,
for every $Q\in \sss_\WSBC(R)\cap\fG\setminus \mathsf{Ex}_\WSBC(R)$, 
 any $\lambda$-good corkscrew for $Q$ can be joined to some $\lambda'$-good corkscrew for $R$ by a $C(\lambda,\eta)$-good Harnack chain, with $\lambda'$ depending on $\lambda,\eta$.
\end{lemma}

\begin{proof} 
 For any $R'\in \DD_{\mu,k_0}\cap\tree_\WSBC(R)$, with $k_0=k_0(\gamma)$, we define $\fH_i(R')$ as in \rf{defhi}, so that 
$$\sss_\WSBC(R)\cap\fG\cap \DD_\mu(R') = \fH_1(R')\cup \fH_2(R').$$
%We choose $\tau=\eta/2$ and $\Gamma$ as in Lemma \ref{keylemma}. 
For each $R'$, we set  
$$\mathsf{Ex}_\WSBC(R') = \bigcup_{i=1}^2\,\,\Bigl\{Q\in\fH_i(R'): \mbox{$\sum_{P\in \fH_i(R')} \mu(P)\leq \tau\,\mu(R')$}\Bigr\}.$$
That is, for fixed $i=1$ or $2$, if $\sum_{P\in \fH_i(R')} \mu(P)\leq \tau\,\mu(R')$, then all the cubes from
$\fH_i(R')$
belong to $\mathsf{Ex}_\WSBC(R')$.
In this way, it is clear that
\begin{equation}\label{eq1pp}
\sum_{P\in\mathsf{Ex}_\WSBC(R')}\mu(P)\leq 2\tau\,\mu(R) = \eta\,\mu(R').
\end{equation}

We claim that the $\lambda$-good corkscrews of cubes from $\sss_\WSBC(R)\cap\fG\cap \DD_\mu(R')\setminus
\mathsf{Ex}_\WSBC(R')$ can be joined to some $\wt\lambda$-good corkscrew for $R'$ by a $\wt C$-good Harnack chain, with $\wt\lambda$ depending on $\lambda,\eta$, and $\wt C$ depending on $\Gamma$ and thus on  $\lambda,\eta$ too. 
Indeed, if $Q\in \fH_i(R')\setminus \mathsf{Ex}_\WSBC(R')$ and $x_Q^i$ is $\lambda$-good corkscrew 
belonging to $V_i$ (we use the notation of Lemma \ref{keylemma} and \ref{lemgeom}), then $\sum_{P\in \fH_i(R')} \mu(P)> \tau\,\mu(R')$ by the definition above
and thus Lemma \ref{keylemma} ensures that $g(p,x_{R'}^i)\geq c(\lambda,\tau)\,\frac{\ell(R')}{\mu(R_0)}$. So $x_{R'}^i$
is a $\wt \lambda$-good corkscrew, which by Lemma \ref{lemgeom}(c) can be joined to $x_Q^i$ by a $\wt C$-good Harnack chain.
In turn, this $\wt\lambda$-good corkscrew for $R'$ can be joined to some $\lambda'$-good corkscrew for $R$ by a $C'$-good Harnack chain, by applying Lemma \ref{lembigcork}
$k_0$ times, with $C'$ depending on $k_0$ and thus on $\lambda$ and $\eta$.

On the other hand, the cubes $Q\in \sss_\WSBC(R)\cap\fG$ which are not contained in any cube $R'\in \DD_{\mu,k_0}\cap\tree_\WSBC(R)$
satisfy $\ell(Q)\geq 2^{-k_0}$, and then, arguing as above, their associated $\lambda$-good corkscrews can be joined
 to some $\lambda'$-good corkscrew for $R$ by a $C'$-good Harnack chain, by applying Lemma \ref{lembigcork}
at most $k_0$ times.
Hence, if we define
$$\mathsf{Ex}_\WSBC(R) = \bigcup_{R'\in\DD_{\mu,k_0}(R)}\mathsf{Ex}_\WSBC(R'),$$
taking into account \rf{eq1pp}, the lemma follows.
\end{proof}

\vv

\begin{proof}[\bf Proof of the Key Lemma \ref{keylemma3}]

We choose $\Gamma=\Gamma(\lambda,\eta)$ as in Lemma \ref{keylemma2} and we consider the associated family $\WSBC(\Gamma)$.
In case that $\tree_\WSBC(R)= \varnothing$,  we set $\mathsf{Ex}(R)=\varnothing$. Otherwise,
we consider the family  
$\mathsf{Ex}_\WSBC(R)$  from Lemma \ref{keylemma2}, and we define
$$\mathsf{Ex}(R) = 
\bigl(\mathsf{Ex}_\WSBC(R)\cap \sss(R)\bigr)\cup\!\! 
\bigcup_{Q\in \mathsf{Ex}_\WSBC(R)\setminus \sss(R)}\!\!\bigl(\ssss(Q)\cap\fG\bigr).$$
It may be useful for the reader to compare the definition above with the partition of $\sss(R)$ in \rf{eqstop221}.
By Lemma \ref{keylemma2} we have
$$\sum_{P\in\mathsf{Ex}(R)}\mu(P)\leq
\sum_{Q\in\mathsf{Ex}_\WSBC(R)}\mu(P)\leq \eta\,\mu(R).$$

Next we show that for every $P\in \sss(R)\cap\mathsf G\setminus \mathsf{Ex}(R)$, 
 any $\lambda$-good corkscrew for $P$ can be joined to some $\lambda'$-good corkscrew for $R$ by a $C(\lambda,\eta)$-good Harnack chain. In fact, if $P\in \sss_\WSBC(R)$, then $P\in \sss_\WSBC(R)\cap\mathsf G\setminus
 \mathsf{Ex}_\WSBC(R)$ since such cube $P$ cannot belong to $\ssss(Q)$ for any $Q\in\sss_\WSBC(R)\setminus \sss(R)$ (recall the partition
\rf{eqstop221}),
  and thus the existence of such Harnack chain is ensured by Lemma \ref{keylemma2}.
On the other hand, if $P\not\in \sss_\WSBC(R)$, then $P$ is contained in some cube $Q(P)\in\sss_\WSBC(R)\setminus\WSBC(\Gamma)$. %Consider some good corkscrew $x_{Q(P)}$ for $Q(P)$ (its existence
%is guarantied by the fact that $Q(P)\in\fG$).
%Then if $x_P$ is some corkscrew for $P$, we claim that this can be joined to some $x_{Q(P)}$ by some
%$C$-good Harnack chain. Indeed, 
Consider the chain $P=S_1\subset S_2\subset\cdots\subset S_m=Q(P)$, so that
each $S_i$ is the parent of $S_{i-1}$. For  $1\leq i\leq m$, choose inductively a big corkscrew $x_i$ for
$S_i$ in such a way that
$x_1$ is at the same side of $L_P$ as the  good $\lambda$ corkscrew $x_P$ for $P$, and $x_{i+1}$ is at the same side of $L_{S_i}$ as $x_i$ for each $i$.
Using that $b\beta(S_i)\leq C\ve\ll1$ for all $i$, it easy to check that the line obtained by joining
the segments $[x_P,x_1]$, $[x_1,x_2]$,\ldots,$[x_{m-1},x_m]$ is a good carrot curve and so gives rise to a good
Harnack chain that joins $x_P$ to $x_m$. It may happen that $x_m$ is not a $\lambda$-good corkscrew. However, since $Q(P)\not\in \WSBC(\Gamma)$, it turns out that $x_m$ can be joined to some $c_3$-good corkscrew $x_{Q(P)}$
for $Q(P)$ by some $C(\Gamma)$-good Harnack chain, with $c_3$ given by \rf{eqgg1} (and thus independent of $\lambda$ and $\eta$), because $Q(P)\in\fG$. Note that since $\lambda\leq c_3$, $x_{Q(P)}$ is also a $\lambda$-good
corkscrew.
In turn, since $Q(P)\not\in\mathsf{Ex}_\WSBC(R)$, $x_{Q(P)}$ can be joined to some $\lambda'$-good corkscrew for $R$ by another $C'(\lambda,\eta)$-good Harnack chain.
Altogether, this shows that $x_P$ can  be connected to some $\lambda'$-good corkscrew for $R$ by a $C''(\lambda,\eta)$-good Harnack chain, which completes the proof of the lemma.
\end{proof}

\vv

Below we will write $\mathsf{Ex}(R,\lambda,\eta)$ instead of $\mathsf{Ex}(R)$ to keep track of the dependence of this family on the parameters $\lambda$ and $\eta$.

\vv
% ***************************************************************************

\section{Proof of the Main Lemma \ref{lemhc}}\label{sec8}

\subsection{Notation}

Recall that by the definition of $G_0^K$ in \rf{defg0k}, $\sum_{R\in\ttt}\chi_R(x) \leq K$ for all $x\in G_0^K$.
For such $x$, let $Q$ be the smallest cube from $\ttt$ that contains $x$, and denote $n_0(x) = -\log_2\ell(Q)$, so that
$Q\in\DD_{\mu,n_0(x)}$.
Next let $N_0\in\Z$ be such that
$$\mu\bigl(\bigl\{x\in G_0^K:n_0(x)\leq N_0-1\bigr\}\bigr) \geq \frac12\,\mu(G_0^K),$$
and denote 
$$\wt G_0^K = \bigl\{x\in G_0^K:n_0(x)\leq N_0-1\bigr\}.$$

Fix $$N=N_0-1,$$ and set
$$\TT_a' = \DD_{\mu,N}(R_0)\cup \ttt^{(N)}_a,$$
and also
$$\TT_b' = \ttt^{(N)}_b\setminus\DD_{\mu,N}(R_0)$$
(recall that $\ttt^{(N)}_a$ and $\ttt^{(N)}_b$ were defined in Section \ref{secnnn}).
So if $R\in\TT_a'\setminus \DD_{\mu,N}(R_0)$, then $\sss^{N}(R)$ coincides the family of sons of $R$, and it $R\in \TT_b'$ this will not be the case, in general. Next we denote by $\TT_a$ and $\TT_b$ the respective subfamilies of cubes from $\TT_a'$ and $\TT_b'$ which intersect $\wt G_0^K$. 

For $j\geq0$, we set
$$\TT_b^j = \Bigl\{R\in\TT_b:\sum_{Q\in\TT_b:Q\supset R} \chi_Q=j\text{ on $R$}\Bigr\}.$$
We also denote
$$\fS_b^j = \bigl\{Q\in\DD_\mu:Q\in\sss^{N}(R)\text{ for some $R\in\TT_b^j$}\bigr\},\qquad \fS_b=\bigcup_j \fS_b^j,$$
and we let $\TT_a^j$ be the subfamily of cubes $R\in\TT_a$ such that there exists some $Q\in\fS_b^{j-1}$
such that $Q\supset R$ and $R$ is not contained in any cube from $\fS_b^{k}$ with $k\geq j$.

% ***************************************************************************
\vv

\subsection{Two auxiliary lemmas}

\vv

\begin{lemma}\label{lem7.1}
The following properties hold for the family $\TT_b^1$:
\begin{itemize}
\item[(a)] The cubes from $\TT_b^1$ are pairwise disjoint and cover  $\wt G_0^K$, assuming
$N_0$ big enough.
\item[(b)] If $R\in\TT_b^1$, then
$\ell(R)\approx_K\ell(R_0)$.
\item[(c)] Given $R\in\DD_\mu(R_0)$ with $\ell(R)\geq c\,\ell(R_0)$ (for example, $R\in\TT_b^1$) and
$\lambda>0$, if $x_R$ is a $\lambda$-good corkscrew point 
for $R$, then there is a $C(\lambda,c)$-good Harnack chain that joins $x_R$ to $p$.
\end{itemize}
\end{lemma}

\begin{proof} 
Concerning the statement (a), the cubes from $\TT_b^1$ are pairwise disjoint by construction. Suppose that
$x\in \wt G_0^K$ is not contained in any cube from $\TT_b^1$. By the definition of the family $\ttt^{N}$, this implies that all the cubes $Q\subset R_0$ with $2^{-N}\ell(R_0)\leq \ell(Q)\leq 2^{-10}\ell(R_0)$ containing $x$ belong to $\TT_a$. However, there are most $K$ cubes $Q$ of this type, which is not possible
if $N$ is taken big enough. So the cubes from $\TT_b^1$ cover $\wt G_0^K$.

The proof of (b) is analogous. Given $R\in\TT_b^1$, all the cubes $Q$ which contain $R$ and
have side length smaller or equal that $2^{-10}\ell(R_0)$ belong to $\TT_a$.
Hence there at most $K-1$ cubes $Q$ of this type, because $\wt G_0^K \cap R\neq\varnothing$. Thus,
$\ell(R)\geq 2^{-K-10}\ell(R_0)$.

The statement (c) is an immediate consequence of (b) and Lemma \ref{lemclosejumps}.
\end{proof}

\vv
\begin{lemma}\label{lem7.2}
Let $Q\in\TT_a^j\cup\TT_b^j$ for some $j\geq 2$ and let $x_Q$ be a $\lambda$-good corkscrew for $Q$, with $\lambda>0$. There exists some constant $\gamma(\lambda,K)>0$ such if $\ell(Q)\leq \gamma(\lambda,K)\,\ell(R_0)$, then there exists some cube
$R\in\fS_b$ such that $R\supset Q$ with a $\lambda'$-good corkscrew $x_R$ for $R$ such that
 $x_R$ can be joined to $x_Q$ by a $C(\lambda,K)$-good Harnack chain, with $\lambda'$ depending on $\lambda$ and $K$.
\end{lemma}

\begin{proof}
We assume $\gamma(\lambda,K)>0$ small enough. Then we can apply Lemma \ref{lemshortjumps2} $K+1$ times
to get cubes $R_1,\ldots,R_{K+1}$ satisfying:
\begin{itemize}
\item $Q\subsetneq R_1\subsetneq R_2\subsetneq\ldots \subsetneq R_{K+1}$ and  $\ell(R_{K+1})\leq 2^{-10}\ell(R_0)$,
\item each $R_j$ has an associated $\lambda'$-good corkscrew $x_{R_i}$ (with $\lambda'$ depending on $\lambda,K$) and there exists a $C(\lambda,K)$-good Harnack chain joining $x_Q$ and $x_{R_1},
\ldots,x_{R_{K+1}}$.
\end{itemize} 
Since $Q\cap\wt G_0^K\neq\varnothing$, at least one of the cubes $R_1,\ldots,R_{K+1}$, say $R_j$, does not belong
to $\ttt$. This implies that $R_j\in\tree^{(N)}(\wt R)$ for some $\wt R\in\TT_b$. Let $R\in\sss^{(N)}(\wt R)$
be the stopping cube that contains $Q$.  Then Lemma \ref{lemgeom} ensures that there is a good Harnack chain
that connects $x_{R_j}$ to some corkscrew $x_R$ for $R$. Notice that $\ell(R_j)\approx_{\lambda,K}\ell(Q)\approx_{\lambda,K}
\ell(R)$ because
$Q\subset R\subset R_j$. This implies that $g(p,x_R)\approx_{K,\lambda} g(p,x_{R_j})\approx_{K,\lambda} g(p,x_Q)$. Further, gathering the Harnack chain that joins $x_Q$ to $x_{\wt R}$ and the one
that joins $x_{R_j}$ to $x_R$, we obtain the good Harnack chain required by the lemma.
\end{proof}

% ***************************************************************************
\vv

\subsection{The algorithm to construct good Harnack chains}

We will construct good Harnack chains that join good corkscrews from ``most'' cubes from $\DD_{\mu,N}$ that intersect $\wt G_0^K$
to good corkscrews from cubes belonging to $R\in\TT_b^1$, and then we will join these latter good corkscrews to
$p$ using the fact that $\ell(R)\approx\ell(R_0)$.
To this end we choose $\eta>0$ such that
$$\eta \leq \frac1{2K}\,\frac{\mu(\wt G_0^K)}{\mu(R_0)},$$
and we
denote
$$m=\max_{x\in \wt G_0^K}\sum_{R\in\TT_b}\chi_R(x)$$
(so that $m\leq K$) and we apply the following algorithm: we set $a_{m+1}=c_3$, so that \rf{eqgg1} ensures that 
for each $Q\in\TT_a\cup\TT_b$ there exists some good $a_{m+1}$-good corkscrew $x_Q$. 
For $j=m,m-1,\ldots,1$, we perform the following procedure:
\vv

\fbox{\begin{minipage}{14cm}
\vv
\begin{enumerate}

\item Join $a_{j+1}$-good corkscrews of cubes $Q$ from $\TT_a^{j+1}\cup \TT_b^{j+1}$ such that $\ell(Q)\leq c_j'\,\ell(R_0)$ to $a_j'$-good corkscrews of cubes $R(Q)$ 
from $\fS_b^1\cup\ldots\cup \fS_b^j$ by $C_j'$-good Harnack chains, with $a_j'\leq a_{j+1}$, so that $R(Q)$ is an ancestor of $Q$. This step can be performed because
of Lemma \ref{lem7.2}, with $c_j'=\gamma(a_{j+1},K)$ in the lemma. The constants $a_j'$, $c_j'$, and $C_j'$ depend on $a_{j+1}$
and $K$.

\item Set 
$$\NC_j = \bigcup_{R\in\TT_b^j} \mathsf{Ex}(R,a_j',\eta),$$
and join $a_j'$-good corkscrews for all cubes $Q\in\fS_b^j\setminus \NC_j$ to $a_j$-good corkscrews for cubes $R(Q)\in\TT_b^j$ by $C_j$-good Harnack chains, with $a_j\leq a_j'$, so that $R(Q)$ is an ancestor of $Q$.  To this end, one applies Lemma \ref{keylemma3}, which ensures the existence
of such Harnack chains connecting $a_j'$-good corkscrew points for cubes from $\fS_b^j \setminus \NC_j$ to
$a_j$-good corkscrew points for cubes from $\TT_b^j$. The constants $a_j$ and $C_j$ depend on $a_{j}'$ and $K$.
\\
\end{enumerate}
\end{minipage}
\vv}
\vspace{0.5cm}

After iterating the procedure above for $j=m,m-1\ldots,1$ and joining some Harnack chains arisen in the different iterations, we will have constructed $C$-good Harnack chains that join  
$a_{m+1}$-good corkscrew points for all cubes $Q\in\TT_a$ not contained in $\bigcup_{j=1}^m\bigcup_{P\in\NC_j}P$
to $a_{1}$-good corkscrews of some ancestors $R(Q)$ belonging either $\TT_b^1$ or, more generally, such that $\ell(R(Q))\gtrsim \ell(R_0)$. The constants $c_j'$, $a_j'$, $a_j$, $C_j$ worsen at
each step $j$. However, this is not harmful because the number of iterations of the procedure is at most $m$, and $m\leq K$.

Denote by $I_N$ the cubes from $\DD_{\mu,N}$ which intersect $\wt G_0^K$ and are not contained in any cube from 
$\{P\in\NC_j:j=1,\ldots m\}$.
By the algorithm above we have constructed good Harnack chains that join $a_{m+1}$-good corkscrew points for all cubes $Q\in I_N$ to some 
 to some $a_1$-good corkscrew for cubes $R(Q)\in\DD_\mu(R_0)$ with $\ell(R(Q))\approx \ell(R_0)$. Also, by applying Lemma \ref{lem7.1} (c) we can connect the $a_1$-good corkscrew for $R(Q)$ to $p$ by a good Harnack chain.
 
Consider now an arbitrary point $x\in \wt G_0^K\cap Q$, with $Q\in I_N$.
By the definition of $\wt G_0^K$ and the choice $N=N_0$, all the cubes $P\in\DD_\mu$ containing $x$ with side length smaller
or equal than $\ell(Q)$ satisfy $b\beta(P)\leq \ve$. Then, by an easy geometric argument (see the proof of Lemma \ref{keylemma3} for a related argument) it is easy to check that there is a good Harnack chain joining any good corkscrew for $Q$ to $x$. Hence, for all the points $x\in \bigcup_{Q\in I_N}Q\cap \wt G_0^K$
there is a good Harnack chain that joins $x$ to $p$.

Finally, observe that, for each $j$, by Lemma \ref{keylemma3},
$$\sum_{P\in\NC_j}\mu(P) = \sum_{R\in\TT_b^j}\sum_{P\in\mathsf{Ex}(R,a_j',\eta)}\mu(P) \leq
\eta \sum_{R\in\TT_b^j}\mu(R) \leq \eta\,\mu(R_0) \leq  \frac1{2K}\,\mu(\wt G_0^K).$$
Therefore,
$$\sum_{j=1}^m\sum_{P\in\NC_j}\mu(P)\leq \frac m{2K}\,\mu(\wt G_0^K)\leq \frac12\,\mu(\wt G_0^K),$$
and thus
$$\sum_{Q\in I_N} \mu(Q)\geq \mu(\wt G_0^K) - \sum_{j=1}^m\sum_{P\in\NC_j}\mu(P) \geq \frac12\,\mu(\wt G_0^K) \approx \mu(R_0).$$
This finishes the proof of the Main Lemma \ref{lemhc}.
\fiproof

\vv

\begin{remark}
Recall that in the arguments above we assumed that $\Omega=\R^{n+1}\setminus \partial\Omega$.
For the general case, we define the auxiliary open set $\wt \Omega = \R^{n+1}\setminus \partial\Omega$, and we apply the arguments above to $\wt \Omega$. Then we will get carrot curves contained in
$\wt\Omega$ that join points from a big piece of $\wt G_0^K$ to $p$.
A quick inspection of the construction above shows that these carrot curves are contained
in the set $\{x\in\wt\Omega: g(p,x)>0\}$, which is a subset of $\Omega$, which implies the
required connectivity condition to conclude the proof of the Main Lemma \ref{lemhc}.
\end{remark}

% ***************************************************************************
% ***************************************************************************
% ***************************************************************************

\vvv

\end{document}